\documentclass{amsart}
\usepackage{amssymb,amsmath}
\usepackage{graphicx}
\theoremstyle{plain}
\newtheorem{thm}{Theorem}[section]

\newtheorem{cor}[thm]{Corollary}
\newtheorem{lem}[thm]{Lemma}
\newtheorem{prop}[thm]{Proposition}

\theoremstyle{definition}
\newtheorem{defn}{Definition}[section]

\theoremstyle{remark}
\newtheorem{rem}{Remark}[section]
\newtheorem{ex}{Example}[section]

\numberwithin{equation}{section}

\newcommand{\ra}{\rightarrow}
\newcommand{\Ra}{\Rightarrow}

\newcommand{\Lra}{\Leftrightarrow}

\begin{document}

\title{Spectral computations for birth and death chains}

\author[G.-Y. Chen]{Guan-Yu Chen$^1$}

\author[L. Saloff-Coste]{Laurent Saloff-Coste$^2$}

\thanks{$^1$Partially supported by NSC grant NSC100-2115-M-009-003-MY2 and NCTS, Taiwan}

\address{$^1$Department of Applied Mathematics, National Chiao Tung University, Hsinchu 300, Taiwan}
\email{gychen@math.nctu.edu.tw}

\thanks{$^2$Partially supported by NSF grant DMS-1004771}

\address{$^2$Malott Hall, Department of Mathematics, Cornell University, Ithaca, NY 14853-4201}
\email{lsc@math.cornell.edu}



\begin{abstract}
We consider the spectrum of birth and death chains on a $n$-path. An iterative scheme is proposed to compute any eigenvalue with exponential convergence rate independent of $n$. This allows one to determine the whole spectrum in order $n^2$ elementary operations. Using the same idea, we also provide a lower bound on the spectral gap, which is of the correct order on some classes of examples.
\end{abstract}

\maketitle

\section{Introduction}\label{s-intro}
Let $G=(V,E)$ be the undirected finite path with vertex set $V=\{1,2,...n\}$ and edge set $E=\{\{i,i+1\}:i=1,2,...,n-1\}$. Given two positive measures $\pi,\nu$ on $V,E$ with $\pi(V)=1$, the Dirichlet form and variance associated with $\nu$ and $\pi$ are defined by
\[
 \mathcal{E}_\nu(f,g):=\sum_{i=1}^{n-1}[f(i)-f(i+1)][g(i)-g(i+1)]\nu(i,i+1)
\]
and
\[
 \text{Var}_\pi(f):=\pi(f^2)-\pi(f)^2,
\]
where $f,g$ are functions on $V$. When convenient, we set $\nu(0,1)=\nu(n,n+1)=0$. The spectral gap of $G$ with respect to $\pi,\nu$ is defined as
\[
 \lambda^G_{\pi,\nu}:=\min\left\{\frac{\mathcal{E}_\nu(f,f)}{\text{Var}_\pi(f)}
 \bigg|f\text{ is non-constant}\right\}.
\]
Let $M^G_{\pi,\nu}$ be a matrix given by $M^G_{\pi,\nu}(i,j)=0$ for $|i-j|>1$ and
\[
 M^G_{\pi,\nu}(i,j)=-\frac{\nu(i,j)}{\pi(i)},\,\forall |i-j|=1,\quad M^G_{\pi,\nu}(i,i)
 =\frac{\nu(i-1,i)+\nu(i,i+1)}{\pi(i)}.
\]
Obviously, $\lambda^G_{\pi,\nu}$ is the smallest non-zero eigenvalue of $M^G_{\pi,\nu}$.

Undirected paths equipped with measures $\pi,\nu$ are closely related to birth and death chains. A birth and death chain on $\{0,1,2,...,n\}$ with birth rate $p_i$, death rate $q_i$ and holding rate $r_i$ is a Markov chain with transition matrix $K$ given by
\begin{equation}\label{eq-bdc}
 K(i,i+1)=p_i,\quad K(i,i-1)=q_i,\quad K(i,i)=r_i,\quad\forall 0\le i\le n,
\end{equation}
where $p_i+q_i+r_i=1$ and $p_n=q_0=0$. Under the assumption of irreducibility, that is, $p_iq_{i+1}>0$ for $0\le i<n$, $K$ has a unique stationary distribution $\pi$ given by $\pi(i)=c(p_0\cdots p_{i-1})/(q_1\cdots q_i)$, where $c$ is the positive constant such that $\sum_{i=0}^n\pi(i)=1$. The smallest non-zero eigenvalue of $I-K$ is exactly the spectral gap of the path on $\{0,1,...,n\}$ with measures $\pi,\nu$, where $\nu(i,i+1)=\pi(i)p_i=\pi(i+1)q_{i+1}$ for $0\le i<n$.

Note that if $\mathbf{1}$ is the constant function of value $1$ and $\psi$ is a minimizer for $\lambda^G_{\pi,\nu}$, then $\psi-\pi(\psi)\mathbf{1}$ is an eigenvector of $M^G_{\pi,\nu}$. This implies that any minimizer $\psi$ for $\lambda^G_{\pi,\nu}$ satisfying $\pi(\psi)=0$ satisfies the Euler-Lagrange equation,
\begin{equation}\label{eq-EL}
 \lambda^G_{\pi,\nu}\pi(i)\psi(i)=[\psi(i)-\psi(i-1)]\nu(i-1,i)+[\psi(i)-\psi(i+1)]\nu(i,i+1),
\end{equation}
for all $1\le i\le n$. Assuming the connectedness of $G$ (i.e., the superdiagonal and subdiagonal entries of $M^G_{\pi,\nu}$ are positive), the rank of $M^G_{\pi,\nu}-\lambda I$ is at least $n-1$. This implies that all eigenvalues of $M^G_{\pi,\nu}$ are simple. See Lemma \ref{l-mat} for an illustration. Observe that, by (\ref{eq-EL}), any non-trivial eigenvector of $M^G_{\pi,\nu}$ has mean $0$ under $\pi$. This implies that all minimizers for the spectral gap are of the form $a\psi+b\mathbf{1}$, where $a,b$ are constants and $\psi$ is a nontrivial solution of (\ref{eq-EL}). In 2009, Miclo obtained implicitly the following result.
\begin{thm}\label{t-Miclo}\cite[Proposition 1]{M09}
If $\psi$ is a minimizer for $\lambda^G_{\pi,\nu}$, then $\psi$ must be monotonic, that is, either $\psi(i)\le\psi(i+1)$ for all $1\le i<n$ or $\psi(i)\ge\psi(i+1)$ for all $1\le i<n$.
\end{thm}

One aim of this paper is to provide a scheme to compute the spectrum of $M^G_{\pi,\nu}$, in particular, the spectral gap. Based on Miclo's observation, it is natural to consider the following algorithm.
\begin{equation}\label{alg-sp}\tag{A1}
\begin{aligned}
    &\text{Choose two positive reals $\lambda_0,a$ in advance and set, for $k=0,1,...$,}\\
    &1.\,\psi_k(1)=-a,\\
    &2.\,\psi_k(i+1)=\psi_k(i)+\frac{\{[\psi_k(i)-\psi_k(i-1)]\nu(i-1,i)
    -\lambda_k\pi(i)\psi_k(i)\}^+}{\nu(i,i+1)},\\
    &\quad\text{for }1\le i<n,\,\text{where }t^+=\max\{t,0\},\\
    &3.\,\lambda_{k+1}=\frac{\mathcal{E}_\nu(\psi_k,\psi_k)}{\textnormal{Var}_{\pi}(\psi_k)}.
\end{aligned}
\end{equation}

The following theorems discuss the behavior of $\lambda_k$.
\begin{thm}[Convergence to the exact value]\label{t-main1}
Referring to \textnormal{(A1)}, if $n=2$, then $\lambda_k=\lambda^G_{\pi,\nu}$ for all $k\ge 1$. If $n\ge 3$, then the sequence $(\lambda_k,\psi_k)$ satisfies
\begin{itemize}
 \item[(1)] If $\lambda_0=\lambda^G_{\pi,\nu}$, then $\lambda_k=\lambda^G_{\pi,\nu}$
for all $k\ge 0$.

 \item[(2)] If $\lambda_0\ne\lambda^G_{\pi,\nu}$, then $\lambda_k>\lambda_{k+1}>\lambda^G_{\pi,\nu}$
for $k\ge 1$.

 \item[(3)] Set $(\lambda^*,\psi^*)=\lim\limits_{k\ra\infty}(\lambda_k,\psi_k)$. Then,
$\lambda^*=\mathcal{E}_\nu(\psi^*,\psi^*)/\textnormal{Var}_\pi(\psi^*)=\lambda^G_{\pi,\nu}$ and $\pi(\psi^*)=0$.
\end{itemize}
\end{thm}

\begin{thm}[Rate of convergence]\label{t-main2}
Referring to {\em Theorem \ref{t-main1}}, there is a constant $\sigma\in(0,1)$ independent of the choice of $(\lambda_0,a)$ such that $0\le \lambda_k-\lambda^G_{\pi,\nu}\le \sigma^{k-1}\lambda_1$ for all $k\ge 1$.
\end{thm}

By Theorem \ref{t-main2}, we know that the sequence $\lambda_k$ generated in (\ref{alg-sp}) converges to the spectral gap exponentially but the rate $(-\log \sigma)$ is undetermined. The following alternative scheme is based on using more information on the spectral gap and will provide convergence at a constant rate.

\begin{equation}\label{alg-sp2}\tag{A2}
\begin{aligned}
    &\textnormal{Choose $a>0,L_0<\lambda^G_{\pi,\nu}<U_0$ in advance and set, for $k=0,1,...$,}\\
    &1.\,\psi_k(1)=-a,\,\lambda_k=\tfrac{1}{2}(L_k+U_k)\\
    &2.\,\psi_k(i+1)=\psi_k(i)+\frac{\{[\psi_k(i)-\psi_k(i-1)]\nu(i-1,i)
    -\lambda_k\pi(i)\psi_k(i)\}^+}{\nu(i,i+1)},\\
    &\quad\text{for }1\le i<n,\,\text{where }t^+=\max\{t,0\},\\
    &3.\,\begin{cases}L_{k+1}=L_k,\,U_{k+1}=\lambda_k&\text{if }\pi(\psi_k)>0\\
    L_{k+1}=\lambda_k,\,U_{k+1}=U_k&\text{if }\pi(\psi_k)<0\\L_{k+1}=U_{k+1}=\lambda_k&\text{if }\pi(\psi_k)=0\end{cases}.
\end{aligned}
\end{equation}

\begin{thm}[Dichotomy method]\label{t-main3}
Referring to \textnormal{(\ref{alg-sp2})}, it holds true that
\[
 0\le\max\{U_k-\lambda^G_{\pi,\nu},\lambda^G_{\pi,\nu}-L_k\}\le (U_0-L_0) 2^{-k},\quad\forall
k\ge 0.
\]
\end{thm}
In Theorem \ref{t-main3}, the convergence to the spectral gap is exponentially fast with explicit rate, $\log 2$. See Remark \ref{r-ini} for a discussion on the choice of $L_0$ and $U_0$.  For higher order spectra, Miclo has a detailed description of the shape of eigenvectors in \cite{M08} and this will motivate the definition of similar algorithms for every eigenvalue in spectrum. See (\ref{alg-evi2}) and Theorem \ref{t-Di} for a generalization of (\ref{alg-sp2}) and Theorem \ref{t-local} for a localized version of Theorem \ref{t-main2}.

The spectral gap is an important parameter in the quantitative analysis of Markov chains. The cutoff phenomenon, a sharp phase transition phenomenon for Markov chains, was introduced by Aldous and Diaconis in early 1980s. It is of interest in many applications. A heuristic conjecture proposed by Peres in 2004 says that the cutoff exists if and only if the product of the spectral gap and the mixing time tends to infinity. Assuming reversibility, this has been proved to hold for $L^p$-convergence with $1<p\le\infty$ in \cite{CSal08}. For the $L^1$-convergence, Ding {\it et al.} \cite{DLP10} prove this conjecture for continuous time birth and death chains. In order to use Peres' conjecture in practice, the orders of the magnitudes of spectral gap and mixing time are required. The second aspect of this paper is to derive a theoretical lower bound on the spectral gap using only the birth and death rates. This lower bound is obtained using the same idea used to analyze the above algorithm. For estimates on the mixing time of birth and death chains, we refer the readers to the recent work \cite{CSal12-3} by Chen and Saloff-Coste. For illustration, we consider several examples of specific interest and show that the lower bound provided here is in fact of the correct order in these examples.

This article is organized as follows. In Section 2, the algorithms in (\ref{alg-sp})-(\ref{alg-sp2}) are explored and proofs for Theorems \ref{t-main1}-\ref{t-main3} are given. In Section 3, the spectrum of $M^G_{\pi,\nu}$ is discussed further and, based on Miclo's work \cite{M08}, Algorithm (\ref{alg-sp2}) is generalized to any specified eigenvalue of $M^G_{\pi,\nu}$. Our method is applicable for paths of infinite length (one-sided) and this is described in Section 4. For illustration, we consider some Metropolis chains and display numerical results of Algorithm (\ref{alg-sp2}) in Section 5. In Section 6, we focus on uniform measures with bottlenecks and determine the correct order of the spectral gap using the theory in Sections 2-3. It is worthwhile to remark that the assumptions in Section 6 can be relaxed using the comparison technique in \cite{DS93-1,DS93-2}. As the work in this paper can also be regarded as a stochastic counterpart of theory of finite Jacobi matrices, we would like to refer the readers to \cite{T96,T00} for a complementary
perspective.

\section{Convergence to the spectral gap}\label{s-gap}
This section is devoted to proving Theorems \ref{t-main1}-\ref{t-main3}. First, we prove Theorem \ref{t-Miclo} in the following form.
\begin{lem}\label{l-polish}
Let $\lambda>0$ and $\psi$ be a non-constant function on $V$. Suppose $(\lambda,\psi)$ solves \textnormal{(\ref{eq-EL})} and $\psi$ is monotonic. Then, $\psi$ is strictly monotonic, that is, either $\psi(i)<\psi(i+1)$ for $1\le i<n$ or $\psi(i)>\psi(i+1)$ for $1\le i<n$.
\end{lem}
\begin{proof}
Obviously, (\ref{eq-EL}) implies that $\pi(\psi)=0$. Without loss of generality, it suffices to consider the case when $\psi(1)<0$ and $\psi(n)>0$. Since $\psi$ is non-constant and $\lambda^G_{\pi,\nu}>0$, we have $\psi(1)<\psi(2)$ and $\psi(n-1)<\psi(n)$. Note that if there are $1<i<j<n$ such that $\psi(i-1)<\psi(i)$, $\psi(j)<\psi(j+1)$ and $\psi(k)=\psi(i)=\psi(j)$ for $i\le k\le j$, then (\ref{eq-EL}) yields
\[
 \lambda^G_{\pi,\nu}\pi(i)\psi(i)=[\psi(i)-\psi(i-1)]\nu(i-1,i)+[\psi(i)-\psi(i+1)]\nu(i,i+1)>0
\]
and
\[
 \lambda^G_{\pi,\nu}\pi(j)\psi(j)=[\psi(j)-\psi(j-1)]\nu(j-1,j)+[\psi(j)-\psi(j+1)]\nu(j,j+1)<0,
\]
a contradiction. Thus, $\psi$ is strictly increasing.
\end{proof}

We note the following corollary.
\begin{cor}\label{c-polish}
Let $(\lambda,\psi)$ be a pair satisfying \textnormal{(\ref{eq-EL})}. Then, $\lambda=\lambda^G_{\pi,\nu}$ if and only if $\psi$ is monotonic.
\end{cor}
\begin{proof}
One direction is obvious from Theorem \ref{t-Miclo}. For the other direction, assume that $\psi$ is monotonic and let $\phi$ be a minimizer for $\lambda^G_{\pi,\nu}$ with $\pi(\phi)=0$. Since $(\lambda,\psi)$ and $(\lambda^G_{\pi,\nu},\phi)$ are solutions to (\ref{eq-EL}), one has
\[
 \lambda\pi(\psi\phi)=\mathcal{E}_\nu(\psi,\phi)=\lambda^G_{\pi,\nu}\pi(\phi\psi).
\]
By Lemma \ref{l-polish}, $\psi$ and $\phi$ are strictly monotonic and this implies $\mathcal{E}_\nu(\psi,\phi)\ne 0$. As a consequence of the above equations, we have $\lambda=\lambda^G_{\pi,\nu}$.
\end{proof}

The following proposition is the key to Theorem \ref{t-main1}.
\begin{prop}\label{p-main1}
Suppose that $(\lambda,\psi)$ satisfies $\lambda>0$, $\psi(1)<0$ and, for $1\le i<n$,
\begin{equation}\label{eq-psigen}
 \psi(i+1)=\psi(i)+\frac{\{[\psi(i)-\psi(i-1)]\nu(i-1,i)-\lambda\pi(i)\psi(i)\}^+}{\nu(i,i+1)},
\end{equation}
where $t^+=\max\{t,0\}$. Then, the following are equivalent.
\begin{itemize}
 \item[(1)] $\mathcal{E}_\nu(\psi,\psi)=\lambda\textnormal{Var}_\pi(\psi)$.

 \item[(2)] $\pi(\psi)=0$.

 \item[(3)] $\lambda=\lambda^G_{\pi,\nu}$.
\end{itemize}
Furthermore, if $n\ge 3$, then any of the above is equivalent to
\begin{itemize}
 \item[(4)] $\mathcal{E}_\nu(\psi,\psi)=\lambda^G_{\pi,\nu}\textnormal{Var}(\psi)$
\end{itemize}
\end{prop}

\begin{rem}\label{r-n=2}
For $n=2$, it is an easy exercise to show that $\lambda^G_{\pi,\nu}=\nu(1,2)/(\pi(1)\pi(2))$. By following the formula in (\ref{eq-psigen}), one has $\psi(2)=\psi(1)[1-\lambda\pi(1)/\nu(1,2)]$, which leads to $\mathcal{E}_\nu(\psi,\psi)/\text{Var}_\pi(\psi)=\lambda^G_{\pi,\nu}$.
\end{rem}

\begin{proof}[Proof of Proposition \ref{p-main1}]
Set $B=\{1\le i\le n|\psi(i)=\psi(n)\}$ and $B^c=\{1,2,...,i_0\}$. Since $\psi(1)<0$ and $\lambda>0$, $\psi(1)<\psi(2)$ and $B^c$ is nonempty. According to (\ref{eq-psigen}), $\psi$ is non-decreasing. Note that if $\psi(i)=\psi(i+1)$, then $\psi(i)\ge 0$ and $\psi(i+2)=\psi(i+1)$. This implies $\psi$ is strictly increasing on $\{1,2,...,i_0+1\}$ and, for $1\le i\le i_0$,
\[
 \lambda\pi(i)\psi(i)=[\psi(i)-\psi(i+1)]\nu(i,i+1)+[\psi(i)-\psi(i-1)]\nu(i-1,i).
\]
Multiplying $\psi(i)$ on both sides and summing over all $i$ in $B^c$ yields
\begin{align}
 \lambda\sum_{i=1}^{i_0}\psi(i)^2\pi(i)&=\sum_{i=1}^{i_0-1}[\psi(i)-\psi(i+1)]^2\nu(i,i+1)\notag\\
 &\qquad+\psi(i_0)[\psi(i_0)-\psi(i_0+1)]\nu(i_0,i_0+1)\notag\\
 &=\mathcal{E}_\nu(\psi,\psi)+\psi(i_0+1)[\psi(i_0)-\psi(i_0+1)]\nu(i_0,i_0+1)\notag\\
 &=\mathcal{E}_\nu(\psi,\psi)+\lambda\psi(n)\sum_{i=1}^{i_0}\psi(i)\pi(i).\notag
\end{align}
This is equivalent to
\begin{equation}\label{eq-Enu}
 \mathcal{E}_\nu(\psi,\psi)=\lambda\text{Var}_\pi(\psi)+\lambda\pi(\psi)[\pi(\psi)-\psi(n)],
\end{equation}
which proves (1)$\Lra$(2).

If $\lambda=\lambda^G_{\pi,\nu}$, then $\psi$ is an eigenvector for $M^G_{\pi,\nu}$ associated to $\lambda^G_{\pi,\nu}$. This proves (3)$\Ra$(2). For (2)$\Ra$(3), assume that $\pi(\psi)=0$. In this case, $\psi$ must be strictly increasing. Otherwise, $\psi(i)=\psi(n)>0$ for $i\in B$ and, according to (\ref{eq-psigen}), this implies
\[
 \lambda\text{Var}_\pi(\psi)>\lambda\sum_{i=1}^{n-1}\pi(i)\psi^2(i)\ge\sum_{i=1}^{n-1}[\psi(i)-\psi(i+1)]^2\nu(i,i+1)=\mathcal{E}(\psi,\psi),
\]
which contradicts (1). As $\psi$ is strictly increasing and $\pi(\psi)=0$, $(\lambda,\psi)$ solves (\ref{eq-EL}). By Corollary \ref{c-polish}, $\lambda=\lambda^G_{\pi,\nu}$.

To finish the proof, it remains to show (4)$\Ra$(3) ((3)$\Ra$(4) is obvious from the equivalence among (1), (2) and (3)). Assume that $\mathcal{E}_\nu(\psi,\psi)=\lambda^G_{\pi,\nu}\text{Var}_\pi(\psi)$. By Lemma \ref{l-polish}, $\psi$ is strictly monotonic and this implies, for $1\le i<n$,
\[
 \lambda\pi(i)\psi(i)=[\psi(i)-\psi(i+1)]\nu(i,i+1)+[\psi(i)-\psi(i-1)]\nu(i-1,i).
\]
As $\psi$ is a minimizer for $\lambda^G_{\pi,\nu}$, one has, for $1\le i\le n$,
\[
 \lambda^G_{\pi,\nu}\pi(i)[\psi(i)-\pi(\psi)]=[\psi(i)-\psi(i+1)]\nu(i,i+1)
 +[\psi(i)-\psi(i-1)]\nu(i-1,i).
\]
If $\lambda\ne\lambda^G_{\pi,\nu}$, the comparison of both systems yields
\[
 \psi(i)=\frac{\lambda^G_{\pi,\nu}\pi(\psi)}{\lambda^G_{\pi,\nu}-\lambda},\quad\forall 1\le i<n.
\]
As $n\ge 3$, $\psi(1)=\psi(2)$, a contradiction! This forces $\lambda=\lambda^G_{\pi,\nu}$, as desired.
\end{proof}

The following is a simple corollary of Proposition \ref{p-main1}, which plays an important role in proving Theorem \ref{t-main3}.
\begin{cor}\label{c-T}
Let $n\ge 3$. For $\lambda>0$, let $\phi_\lambda$ be the vector generated by \textnormal{(\ref{eq-psigen})} with $\phi(1)<0$. Then, $(\lambda-\lambda^G_{\pi,\nu})\pi(\phi_\lambda)>0$ for $\lambda>0$ and $\lambda\ne\lambda^G_{\pi,\nu}$.
\end{cor}
\begin{proof}
Without loss of generality, we fix $\phi_\lambda(1)=-1$ for all $\lambda>0$. Set $T(\lambda)=\pi(\phi_\lambda)$. To prove this corollary, it suffices to show that
\[
 T(\lambda)\begin{cases}<0&\text{if }\lambda<\lambda^G_{\pi,\nu}\\>0&\text{if }\lambda>\lambda^G_{\pi,\nu}\end{cases}.
\]
For $\lambda>0$, define $L(\lambda):=\mathcal{E}_\nu(\phi_\lambda,\phi_\lambda)/\text{Var}_\pi(\phi_\lambda)$. By (\ref{eq-Enu}), one has
\begin{equation}\label{eq-L}
 L(\lambda)-\lambda=\frac{\lambda T(\lambda)[\pi(\phi_\lambda)-\phi_\lambda(n)]}{\text{Var}_\pi(\phi_\lambda)}.
\end{equation}
Since $\phi_\lambda$ is non-constant, $\pi(\phi_\lambda)<\phi_\lambda(n)$. This implies $T(\lambda)<0$ for $\lambda\in(0,\lambda^G_{\pi,\nu})$.

For $\lambda>\lambda^G_{\pi,\nu}$, set $I=(\lambda^G_{\pi,\nu},\infty)$. By Proposition \ref{p-main1}, $T(\lambda)=0$ if and only if $\lambda=\lambda^G_{\pi,\nu}$. By the continuity of $T$, this implies either $T(I)\subset(-\infty,0)$ or $T(I)\subset(0,\infty)$. In the case $T(I)\subset(-\infty,0)$, one has $L(\lambda)>\lambda$ for $\lambda\in I$. As $L(I)$ is bounded, $L^k(\lambda)$ is convergent with limit $\widetilde{\lambda}>\lambda^G_{\pi,\nu}$ and this yields
\[
 0=\lim_{k\ra\infty}[L^{k+1}(\lambda)-L^k(\lambda)]
 =\frac{\widetilde{\lambda}T(\widetilde{\lambda})[\pi(\phi_{\widetilde{\lambda}})
 -\phi_{\widetilde{\lambda}}(n)]}{\text{Var}_\pi(\phi_{\widetilde{\lambda}})}>0,
\]
a contradiction. Hence, $T(\lambda)>0$ for $\lambda>\lambda^G_{\pi,\nu}$.
\end{proof}

\begin{proof}[Proof of Theorem \ref{t-main1}]
The proof for $n=2$ is obvious from a direct computation and we deal with the case $n\ge 3$, here. By the equivalence of Proposition \ref{p-main1} (3)-(4), if $\lambda_0=\lambda^G_{\pi,\nu}$, then $\lambda_k=\lambda^G_{\pi,\nu}$ for all $k\ge 1$. If $\lambda_0\ne\lambda^G_{\pi,\nu}$, then $\lambda_k>\lambda^G_{\pi,\nu}$ for $k\ge 1$. Note that $(\lambda_k,\psi_k)$ solves the system in (\ref{eq-psigen}). By (\ref{eq-Enu}), this implies
\[
 \lambda_{k+1}-\lambda_k=\frac{\lambda_k\pi(\psi_k)[\pi(\psi_k)-\psi_k(n)]}{\text{Var}_\pi(\psi_k)},
 \quad\forall k\ge 0.
\]
The strict monotonicity of $\lambda_k$ in (2) comes immediately from Corollary \ref{c-T}. In (3), the continuity of (\ref{eq-psigen}) in $\lambda$ implies that $(\lambda^*,\psi^*)$ is a solution to (\ref{eq-psigen}) and $\mathcal{E}_\nu(\psi^*,\psi^*)=\lambda^*\text{Var}(\psi^*)$. By Proposition \ref{p-main1}, $\lambda^*=\lambda^G_{\pi,\nu}$ and $\pi(\psi^*)=0$, as desired.
\end{proof}

\begin{proof}[Proof of Theorem \ref{t-main2}]
Recall the notation in the proof of Corollary \ref{c-T}: For $\lambda>0$, let $\phi_\lambda$ be the function defined by (\ref{eq-psigen}) and $L(\lambda)=\mathcal{E}_\nu(\phi_\lambda,\phi_\lambda)/\text{Var}_\pi(\phi_\lambda)$. By (\ref{eq-Enu}) and Corollary \ref{c-T}, $L(\lambda)\in(\lambda^G_{\pi,\nu},\lambda)$ for $\lambda>\lambda^G_{\pi,\nu}$. As $L$ is bounded, Theorem \ref{t-main2} follows from Lemma \ref{l-conv}.
\end{proof}

\begin{proof}[Proof of theorem \ref{t-main3}]
Immediate from Corollary \ref{c-T}.
\end{proof}

In the end of this section, we use the following proposition to find how the shape of the function $\psi$ in (\ref{eq-psigen}) evolves with $\lambda$. In Proposition \ref{p-shape}, we set $\phi_\lambda=\psi$ when $\psi$ is given by (\ref{eq-psigen}). It is easy to see from (\ref{eq-psigen}) that $\phi_\lambda$ is strictly increasing before some constant, say $i_0=i_0(\lambda)$, and then stays constant equal to $\phi_\lambda(i_0)$ after $i_0$. The proposition shows how the constant $i_0(\lambda)$ evolves.
\begin{prop}\label{p-shape}
For $\lambda>0$, let $\phi_\lambda$ be the function generated by \textnormal{(\ref{eq-psigen})} with $\phi_\lambda(1)=-1$ and, for $1\le i\le n$, set $T_i(\lambda)=\sum_{j=1}^i\phi_\lambda(i)\pi(i)$. For $1\le i<n$, let
\[
a_i(\lambda)=1+\pi(i+1)/\pi(i)-\lambda\pi(i+1)/\nu(i,i+1),
\]
\begin{equation}\label{eq-ail}
    A_i(\lambda)=\left(\begin{array}{cccccc}
    a_1(\lambda)&1&0&0&\cdots&0\\
    \frac{\pi(3)}{\pi(2)}&a_{2}(\lambda)&1&0&&\vdots\\
    0&\frac{\pi(4)}{\pi(3)}&a_{3}(\lambda)&\ddots&\ddots&\vdots\\
    0&0&\ddots&\ddots&\ddots&0\\
    \vdots&&\ddots&\ddots&a_{i-1}(\lambda)&1\\
    0&\cdots&\cdots&0&\frac{\pi(i+1)}{\pi(i)}&a_i(\lambda)\end{array}\right),
\end{equation}
and let $\lambda^{(i)}$ be the smallest root of $\det A_i(\lambda)=0$. Then,
\begin{itemize}
\item[(1)] $\lambda^G_{\pi,\nu}=\lambda^{(n-1)}<\lambda^{(n-2)}<\cdots<\lambda^{(1)}$.

\item[(2)] $\phi_\lambda(i)<\phi_\lambda(i+1)=\phi_\lambda(i+2)$ for $\lambda\in[\lambda^{(i)},\lambda^{(i-1)})$ and $1\le i\le n-2$, where $\lambda^{(0)}:=\infty$.

\item[(3)] $\phi_\lambda(n-1)<\phi_\lambda(n)$ for $\lambda\in(0,\lambda^{(n-2)})$.
\end{itemize}
In particular, $T_{i+1}(\lambda)=-\pi(1)\det A_i(\lambda)$ for $\lambda\in(0,\lambda^{(i-1)})$ and $(\lambda-\lambda^{(i)})T_{i+1}(\lambda)>0$ for $\lambda\in(0,\lambda^{(i)})\cup(\lambda^{(i)},\infty)$ with $1\le i\le n-1$.
\end{prop}
\begin{proof}
By Lemma \ref{l-matt}, $\lambda^{(1)}>\lambda^{(2)}>\cdots>\lambda^{(n-1)}>0$ and, for $1\le i\le n-1$,
\begin{equation}\label{eq-ail2}
 \det A_i(\lambda)\begin{cases}>0&\forall \lambda\in(-\infty,\lambda^{(i)})\\<0&\forall \lambda\in(\lambda^{(i)},\lambda^{(i-1)})\end{cases},
\end{equation}
where $\lambda^{(0)}=\infty$. Note that if $T_i(\lambda)<0$ for some $1\le i\le n-1$, then
\[
 \phi_\lambda(j+1)=\phi_\lambda(j)+\frac{[\phi_\lambda(j)-\phi_\lambda(j-1)]\nu(j-1,j)
 -\lambda\pi(j)\phi_\lambda(j)}{\nu(j,j+1)},\quad\forall 1\le j\le i.
\]
This implies
\begin{equation}\label{eq-phil}
 \phi_\lambda(\ell+1)=\phi_\lambda(\ell)-\frac{\lambda}{\nu(\ell,\ell+1)}\sum_{j=1}^\ell
 \pi(j)\phi_\lambda(j),\quad\forall 1\le\ell\le i.
\end{equation}
Multiplying $\pi(\ell+1)$ and adding up $T_{\ell}(\lambda)$ yields
\[
 T_{\ell+1}(\lambda)=a_\ell(\lambda)T_\ell(\lambda)-\frac{\pi(\ell+1)}{\pi(\ell)}T_{\ell-1}(\lambda),
 \quad\forall 1\le \ell\le i.
\]
From the above discussion, we conclude that if $T_i(\lambda)<0$, then
\begin{equation}\label{eq-TA}
 T_{\ell+1}(\lambda)=-\pi(1)\det A_\ell(\lambda),\quad \forall 1\le\ell\le i.
\end{equation}
When $\ell=i-1$, (\ref{eq-ail2}) implies $\det A_{i-1}(\lambda)>0$ for $\lambda<\lambda^{(i-1)}$. By the continuity of $T_i$ and $\det A_{i-1}$, if there is some $\lambda<\lambda^{(i-1)}$ such that $T_i(\lambda)<0$, then $T_i(\lambda)=-\pi(1)\det A_{i-1}(\lambda)$ for $\lambda<\lambda^{(i-1)}$. As a consequence of (\ref{eq-TA}) with $\ell=i$, this will imply $T_{i+1}(\lambda)=-\pi(1)\det A_i(\lambda)$ for $\lambda<\lambda^{(i-1)}$. Hence, it remains to show that $T_i(\lambda)<0$ for some $\lambda<\lambda^{(i-1)}$. To see this,
according to Corollary \ref{c-T}, one can choose a constant $\widetilde{\lambda}<\min\{\lambda^G_{\pi,\nu},\lambda^{(i-1)}\}$ such that $T_{n-1}(\widetilde{\lambda})<0$. Since $\phi_\lambda(i)$ is non-decreasing in $i$, we obtain $T_i(\widetilde{\lambda})<0$, as desired. This proves $T_{i+1}(\lambda)=-\pi(1)\det A_i(\lambda)$ for $\lambda<\lambda^{(i-1)}$. In particular, $T_n(\lambda)=-\pi(1)\det A_{n-1}(\lambda)$ for $\lambda<\lambda^{(n-2)}$. By Corollary \ref{c-T}, we have $\lambda^{(n-1)}=\lambda^G_{\pi,\nu}$. This proves Proposition \ref{p-shape} (1).

Next, observe that, for $\lambda\in(\lambda^{(i)},\lambda^{(i-1)})$,
\[
 \sum_{j=1}^{i+1}\pi(j)\phi_\lambda(j)=T_{i+1}(\lambda)>0,\quad \sum_{j=1}^i\pi(j)\phi_\lambda(j)=T_i(\lambda)<0.
\]
By (\ref{eq-phil}), it is easy to see that $[\phi_\lambda(i+1)-\phi_\lambda(i)]\nu(i,i+1)=-\lambda T_i(\lambda)>0$ and
\begin{align}
 &[\phi_\lambda(i+2)-\phi_\lambda(i+1)]\nu(i+1,i+2)\notag\\
 =&\left\{[\phi_\lambda(i+1)-\phi_\lambda(i)]\nu(i,i+1)
 -\lambda\pi(i+1)\phi_\lambda(i+1)\right\}^+\notag\\
=&\{-\lambda T_{i+1}(\lambda)\}^+=0.\notag
\end{align}
This proves Proposition \ref{p-shape} (2). To prove Proposition \ref{p-shape} (3), we use (1) to derive
\[
 T_{n-1}(\lambda)=-\pi(1)\det A_{n-2}(\lambda)<0,\quad\forall \lambda\in(0,\lambda^{(n-2)}).
\]
Using (\ref{eq-phil}), this implies $\phi_\lambda(n-1)<\phi_\lambda(n)$.
The last part of Proposition \ref{p-shape} follows easily from (\ref{eq-ail2}) and the  fact that
\[
 T_i(\lambda)\ge 0\Ra T_{i+1}(\lambda)>0  \mbox{ and }
 T_i(\lambda)\le 0\Ra T_{i-1}(\lambda)<0.
\]
\end{proof}

\begin{rem}\label{r-ini}
In Proposition \ref{p-shape}, if $\lambda>\lambda^{(1)}=\nu(1,2)[\pi(1)^{-1}+\pi(2)^{-1}]$, then $\phi_\lambda(i)=\phi_\lambda(2)$ for $i=2,...,n$. Note that, for $\lambda\ge\lambda^{(1)}$, $\phi_\lambda(2)=-1+\lambda\pi(1)/\nu(1,2)$ and
\[
 \pi(\phi_\lambda)=-1+\frac{\lambda\pi(1)(1-\pi(1))}{\nu(1,2)},\quad
\text{Var}_\pi(\phi_\lambda)=\frac{\lambda^2\pi(1)^3(1-\pi(1))}{\nu(1,2)^2}.
\]
By (\ref{eq-L}), this leads to $L(\lambda)=\nu(1,2)/[\pi(1)(1-\pi(1)]$ for $\lambda\ge\lambda^{(1)}$. In the case $n=2$, it is clear that $\nu(1,2)/[\pi(1)(1-\pi(1)]=\nu(1,2)[\pi(1)^{-1}+\pi(2)^{-1}]=\lambda^G_{\pi,\nu}$.
\end{rem}

\section{Convergence to other eigenvalues}\label{s-framework}
In this section, we generalize the algorithms (\ref{alg-sp}) and (\ref{alg-sp2})
so that they can be applied for the computation to any specified eigenvalue.

\subsection{Basic setup and fundamental results}
Recall that $G$ is a graph with vertex set $V=\{1,2,...,n\}$ and edge set $E=\{\{i,i+1\}|i=1,2,...,n-1\}$. Given two positive measures $\pi,\nu$ on $V,E$ with $\pi(V)=1$, let $M^G_{\pi,\nu}$ be a $n$-by-$n$ matrix defined in the introduction and given by
\begin{equation}\label{eq-M}
 M^G_{\pi,\nu}(i,j)=\begin{cases}-\nu(i,j)/\pi(i)&\text{if }|i-j|=1\\ [\nu(i-1,i)+\nu(i,i+1)]/\pi(i)&\text{if }j=i\\0&\text{if }|i-j|>1\end{cases}.
\end{equation}
Since $\nu$ is positive everywhere and $M^G_{\pi,\nu}$ is tridiagonal, all eigenvalues of $M^G_{\pi,\nu}$ have algebraic multiplicity $1$. Throughout this section, let $\{\lambda^G_0<\lambda_1^G<\cdots<\lambda_{n-1}^G\}$ denote the eigenvalues of $M^G_{\pi,\nu}$ with associated $L^2(\pi)$-normalized eigenvectors $\zeta_0=\mathbf{1},\zeta_2,...,\zeta_{n-1}$. Clearly, $\lambda^G_0=0$, $\lambda_1^G=\lambda^G_{\pi,\nu}$ and, for $1\le k\le n$,
\begin{equation}\label{eq-EL2}
 \lambda_i^G\zeta_i(k)\pi(k)=[\zeta_i(k)-\zeta_i(k-1)]\nu(k-1,k)+[\zeta_i(k)-\zeta_i(k+1)]\nu(k,k+1).
\end{equation}
Let $1\le i\le n-1$. As $\zeta_i$ is non-constant, it is clear that $\zeta_i(1)\ne\zeta_i(2)$ and $\zeta_i(n-1)\ne\zeta_i(n)$. Moreover, if $\zeta_i(k)=\zeta_i(k+1)$ for some $1<k<n$, then $\zeta_i(k)\ne\zeta_i(k-1)$ and $\zeta_i(k+1)\ne\zeta_i(k+2)$. Gantmacher and Krein \cite{GK37} showed that there are exactly $i$ sign changes for $\zeta_i$ with $1\le i\le n$. Miclo \cite{M08} gives a detailed description on the shape of $\zeta_i$ as follows.

\begin{thm}\label{t-Miclo2}
For $1\le i\le n-1$, let $\zeta_i$ be an eigenvector associated to the $i$th smallest non-zero eigenvalue of the matrix in \textnormal{(\ref{eq-M})} with $\zeta_i(1)<0$. Then, there are $1=a_1<b_1\le a_2<b_2\le\cdots\le a_i<b_i=n$ with $a_{j+1}-b_j\in\{0,1\}$ such that $\zeta_i$ is strictly increasing on $[a_j,b_j]$ for odd $j$ and is strictly decreasing on $[a_j,b_j]$ for even $j$, and $\zeta_i(a_{j+1})=\zeta_i(b_j)$ for $1\le j<i$.
\end{thm}

In the following, we make some analysis related to the Euler-Lagrange equations in (\ref{eq-EL2}).

\begin{defn}\label{def-type}
Fix $n\ge 1$ and let $f$ be a function on $\{1,2,...,n\}$. For $1\le i\le n-1$, $f$ is called ``Type $i$'' if there are $1=a_1<b_1\le a_2<b_2\le\cdots\le a_i<b_i\le n$ satisfying $a_{j+1}-b_j\in\{0,1\}$ such that
\begin{itemize}
\item[(1)] $f$ is strictly monotonic on $[a_j,b_j]$ for $1\le j\le i$.

\item[(2)] $[f(a_j)-f(a_j+1)][f(a_{j+1})-f(a_{j+1}+1)]<0$ for $1\le j<i$.

\item[(3)] $f(a_{j+1})=f(b_j)$, for $1\le j<i$, and $f(k)=f(b_i)$, for $b_i\le k\le n$.
\end{itemize}
The points $a_j,b_j$ will be called ``peak-valley points'' in this paper.
\end{defn}
\begin{rem}
Note that the difference between Definition \ref{def-type} and Theorem \ref{t-Miclo2} is the requirement $b_i\le n$, instead of $b_i=n$. By Theorem \ref{t-Miclo2}, any eigenvector associated to the $i$th smallest non-zero eigenvalue of the matrix in (\ref{eq-M}) must be of type $i$ with $b_i=n$.
\end{rem}

\begin{defn}\label{def-xil}
Let $\pi,\nu$ be positive measures on $V,E$ with $\pi(V)=1$. For $\lambda\in\mathbb{R}$, let $\xi_\lambda$ be a function on $\{1,2,...,n\}$ defined by $\xi_\lambda(1)=-1$ and, for $1\le k<n$,
\[
 \xi_\lambda(k+1)=\xi_\lambda(k)+\frac{[\xi_\lambda(k)-\xi_\lambda(k-1)]\nu(k-1,k)
-\lambda\pi(k)\xi_\lambda(k)}{\nu(k,k+1)}.
\]
\end{defn}

\begin{rem}\label{r-xi1}
Note that $\xi_0=-\mathbf{1}$ and, for $\lambda<0$, $\xi_\lambda$ is strictly decreasing and of type $1$. For $\lambda>0$, if $\xi_\lambda(k-1)<\xi_\lambda(k)=\xi_\lambda(k+1)$, then $\xi_\lambda(k)>0$ and this implies $\xi_\lambda(k+2)<\xi_\lambda(k+1)$. Similarly, if $\xi_\lambda(k-1)>\xi_\lambda(k)=\xi_\lambda(k+1)$, then $\xi_\lambda(k)<0$ and $\xi_\lambda(k+2)>\xi_\lambda(k+1)$. Thus, $\xi_\lambda$ must be of type $i$ for some $1\le i\le n-1$.
\end{rem}

\begin{lem}\label{l-xi}
For $\lambda>0$, let $\xi_\lambda$ be the function in Definition \ref{def-xil}. Suppose that $\xi_\lambda$ is of type $i$ with $1\le i\le n-1$.
\begin{itemize}
\item [(1)] If $\xi_\lambda(n-1)\ne\xi_\lambda(n)$, then there is $\epsilon>0$ such that $\xi_{\lambda+\delta}$ is of type $i$ for $-\epsilon<\delta<\epsilon$.

\item[(2)] If $\xi_\lambda(n-1)=\xi_\lambda(n)$, then there is $\epsilon>0$ such that $\xi_{\lambda+\delta}$ is of type $i+1$ and $\xi_{\lambda-\delta}$ is of type $i$ for $0<\delta<\epsilon$.
\end{itemize}
\end{lem}
\begin{proof}
Let $a_j,b_j$ be the peak-valley points of $\xi_\lambda$. By the continuity of $\xi_\lambda$ in $\lambda$ and Remark \ref{r-xi1}, one can choose $\epsilon>0$ such that, for $\delta\in(-\epsilon,\epsilon)$, $\xi_{\lambda+\delta}$ remains strictly monotonic on $[a_j,b_j]$ for $j=1,...,i$ and
\[
 [\xi_{\lambda+\delta}(b_j-1)-\xi_{\lambda+\delta}(b_j)]
[\xi_{\lambda+\delta}(a_{j+1}+1)-\xi_{\lambda+\delta}(a_{j+1})]>0,
\]
for $1\le j<i$. In (1), $b_i=n$. Fix $\delta\in(-\epsilon,\epsilon)$ and set $a_1'=a_1=1$, $b_i'=b_i=n$.  For $1<j<i$, set
\[
 \begin{cases}b_j'=a_{j+1}'=b_j &\text{if } [\xi_{\lambda+\delta}(b_j-1)-\xi_{\lambda+\delta}(b_j)][\xi_{\lambda+\delta}(b_j)-\xi_{\lambda+\delta}(a_{j+1})]<0\\
  b_j'=a_{j+1}'=a_{j+1} &\text{if }[\xi_{\lambda+\delta}(b_j-1)-\xi_{\lambda+\delta}(b_j)][\xi_{\lambda+\delta}(b_j)-\xi_{\lambda+\delta}(a_{j+1})]>0\\
  b_j'=b_j,\,a_{j+1}'=a_{j+1}&\text{if }[\xi_{\lambda+\delta}(b_j-1)-\xi_{\lambda+\delta}(b_j)][\xi_{\lambda+\delta}(b_j)-\xi_{\lambda+\delta}(a_{j+1})]=0
 \end{cases}.
\]
Clearly, $\xi_{\lambda+\delta}$ is of type $i$ with peak-valley points $a_j',b_j'$. This proves Lemma \ref{l-xi} (1).

For part (2), we consider $i\le n-2$ and $b_i=n-1$. By similar argument as before, one can choose $\epsilon>0$ such that the restriction of $\xi_{\lambda+\delta}$ to $\{1,2,...,n-1\}$ is of type $i$ for $\delta\in(-\epsilon,\epsilon)$. To finish the proof, it remains to compare $\xi_{\lambda+\delta}(n-1)$ and $\xi_{\lambda+\delta}(n)$. Recall that $T_j(\lambda)=\sum_{k=1}^j\xi_\lambda(k)\pi(k)$ as in the proof for Proposition \ref{p-shape}. Using a similar reasoning as for (\ref{eq-TA}), one shows that $T_{i+1}(\lambda)=-\pi(1)\det A_i(\lambda)$ for $1\le i<n$, where $A_i(\lambda)$ is the matrix in (\ref{eq-ail}). This implies that the non-zero eigenvalues of $M^G_{\pi,\nu}$, say $\lambda^G_1,...,\lambda^G_{n-1}$, are the roots of $\det A_{n-1}(\lambda)=0$. As a consequence of Lemma \ref{l-matt}, $\det A_{n-2}(\lambda)=0$ has exactly $n-2$ distinct roots, say $\alpha_1<\alpha_2<\cdots<\alpha_{n-1}$, and they satisfy the interlacing property $\lambda^G_j<\alpha_j<\lambda^G_{j+1}$ for $1\le j\le n-2$. Note that $\det A_{n-2}(\lambda)$ and $\det A_{n-1}(\lambda)$ tend to infinity as $-\lambda$ tends to infinity. This leads to the fact that if $\det A_{n-2}(\lambda)=0$ and $\det A_{n-1}(\lambda)<0$, then $\det A_{n-2}(\cdot)$ is strictly decreasing in a neighborhood of $\lambda$. If $\det A_{n-2}(\lambda)=0$ and $\det A_{n-1}(\lambda)>0$, then $\det A_{n-2}(\cdot)$ is strictly increasing in a neighborhood of $\lambda$.

Back to the proof of (2). Suppose that $\xi_\lambda(n-2)<\xi_\lambda(n-1)$. By Remark \ref{r-xi1}, it is easy to check that  $T_{n-1}(\lambda)=0$ and $T_n(\lambda)>0$ or, equivalently, $\det A_{n-2}(\lambda)=0$ and $\det A_{n-1}(\lambda)<0$. According to the conclusion in the previous paragraph, we can find $\epsilon>0$ such that $\det A_{n-2}(\cdot)$ is strictly decreasing on $(\lambda-\epsilon,\lambda+\epsilon)$, which yields
\[
 \xi_{\lambda+\delta}(n)=\xi_{\lambda+\delta}(n-1)-\frac{(\lambda+\delta)T_{n-1}(\lambda+\delta)}
{\nu(n-1,n)}\begin{cases}<\xi_{\lambda+\delta}(n-1)&\text{if }0<\delta<\epsilon\\>\xi_{\lambda+\delta}(n-1)&\text{if }-\epsilon<\delta<0\end{cases}.
\]
This gives the desired property in Lemma \ref{l-xi} (2). The other case, $\xi_\lambda(n-2)>\xi_\lambda(n-1)$, can be proved in the same way and we omit
the details.
\end{proof}

The following proposition characterizes the shape of $\xi_\lambda$ for $\lambda>0$.

\begin{prop}\label{p-xi}
For $\lambda>0$, let $\xi_\lambda$ be the function in {\em Definition
\ref{def-xil}}. Let $\lambda^G_1<\cdots<\lambda^G_{n-1}$ be non-zero eigenvalues of $M^G_{\pi,\nu}$ in \textnormal{(\ref{eq-M})} and $\alpha_1<\cdots<\alpha_{n-2}$ be zeros of $\det A_{n-2}(\lambda)$, where $A_{n-2}(\cdot)$ is the matrix in \textnormal{(\ref{eq-ail})}. Then,
\begin{itemize}
\item[(1)] $\lambda^G_j<\alpha_j<\lambda^G_{j+1}$, for $1\le j\le n-2$.

\item[(2)] $\xi_\lambda$ is of type $j$ for $\lambda\in(\alpha_{j-1},\alpha_j]$ and $1\le j\le n-1$, where $\alpha_0:=0$ and $\alpha_{n-1}:=\infty$.
\end{itemize}
\end{prop}
\begin{proof}
(1) is immediate from Lemma \ref{l-matt}. For (2), note that $\alpha_i$ is an eigenvalue of the submatix of $M^G_{\pi,\nu}$ obtained by removing the $n$th row and column. This implies  $\xi_{\alpha_i}(n-1)=\xi_{\alpha_i}(n)$ for $i=1,...,n-2$ and $\xi_\lambda(n-1)\ne\xi_\lambda(n)$ for $\lambda>0$ and $\lambda\notin\{\alpha_1,...,\alpha_{n-2}\}$. By Lemma \ref{l-xi}, $\xi_{\lambda}$ is of type $i$ for $\alpha_{i-1}<\lambda\le\alpha_i$.
\end{proof}

Given $\lambda>0$, the above proposition provides a simple criterion to determine to which of the intervals $(\alpha_j,\alpha_{j+1}]$ $\lambda$ belongs to, that is, the type of $\xi_\lambda$. However, knowing the type of $\xi_\lambda$ is not sufficient to determine whether $\lambda$ is bigger or smaller than $\lambda^G_i$. We need the following remark.

\begin{rem}\label{r-xil}
Using the same argument as the proof of Proposition \ref{p-shape}, one can show that $\pi(\xi_\lambda)=-\pi(1)\det A_{n-1}(\lambda)$, where $A_{n-1}(\lambda)$ is the matrix in (\ref{eq-ail}). Clearly, $\pi(\xi_\lambda)$ has zeros $\lambda^G_1,...,\lambda^G_{n-1}$ and tends to minus infinity as $\lambda$ tends to minus infinity. This implies that $\pi(\xi_\lambda)<0$, for $\lambda<\lambda^G_1$, and
\[
 \pi(\xi_\lambda)>0\quad \forall \lambda\in(\lambda^G_{2i-1},\lambda^G_{2i}),\quad
\pi(\xi_\lambda)<0\quad \forall \lambda\in(\lambda^G_{2i},\lambda^G_{2i+1}),
\]
for $i\ge 1$, where $\lambda^G_n:=\infty$.
\end{rem}

As a consequence of Proposition \ref{p-xi} and Remark \ref{r-xil}, we obtain the following dichotomy algorithm, which is a generalization of (\ref{alg-sp2}). Let $1\le i\le n-1$.
\begin{equation}\label{alg-evi2}\tag{D$i$}
\begin{aligned}
    &\text{Choose positive reals $L_0<\lambda^G_i<U_0$ and set, for $\ell=0,1,...$,}\\
    &1.\,\xi_{\lambda_\ell}\text{ be the function generated by $\lambda_\ell=(L_\ell+U_\ell)/2$ in Definition \ref{def-xil}},\\
    &2.\,\text{According to Definition \ref{def-type}, set}\\
&\,\,\begin{cases}L_{\ell+1}=L_\ell,\, U_{\ell+1}=\lambda_\ell&\text{if $\xi_{\lambda_\ell}$ is of type $j$ with $j>i$,}\\&
 \text{or if $\xi_{\lambda_\ell}$ is of type $i$ and $(-1)^{i-1}\pi(\xi_{\lambda_\ell})>0$}\\U_{\ell+1}=U_{\ell},\,L_{\ell+1}=\lambda_\ell&\text{if $\xi_{\lambda_\ell}$ is of type $j$ with $j<i$,}\\&\text{or if $\xi_{\lambda_\ell}$ is of type $i$ and $(-1)^{i-1}\pi(\xi_{\lambda_\ell})<0$}\\
 L_{\ell+1}=U_{\ell+1}=\lambda_\ell&\text{if $\xi_{\lambda_\ell}$ is of type $i$ and $\pi(\xi_{\lambda_\ell})>0$}\end{cases}.
\end{aligned}
\end{equation}

\begin{thm}\label{t-Di}
Referring to \textnormal{(\ref{alg-evi2})},
\[
 0\le\max\{U_\ell-\lambda^G_i,\lambda^G_i-L_\ell\}\le(U_0-L_0)2^{-\ell},\quad\forall \ell\ge 0.
\]
\end{thm}
\begin{proof}
Immediate from Proposition \ref{p-xi} and Remark \ref{r-xil}.
\end{proof}

Proposition \ref{p-xi} (2) bounds the eigenvalues using the shape of $\xi_\lambda$ generated from one end point. We now introduce some other criteria to bound eigenvalues using the shape of $\xi_\lambda$ from either boundary point. Those results will be used to prove Theorem \ref{t-lower}.

\begin{prop}\label{p-twosided}
For $\lambda>0$, let $\xi_\lambda$ be the function in Definition \ref{def-xil} and $\widetilde{\xi}_\lambda$ be a function given by
\[
 \widetilde{\xi}_\lambda(k-1)=\widetilde{\xi}_\lambda(k)+\frac{[\widetilde{\xi}_\lambda(k)-\widetilde{\xi}_\lambda(k+1)]\nu(k,k+1)-\lambda\pi(k)\widetilde{\xi}_\lambda(k)}{\nu(k-1,k)},
\]
for $k=n,n-1,...,2$ with $\widetilde{\xi}_\lambda(n)=-1$. Let $\lambda^G_0<\cdots<\lambda^G_{n-1}$ be eigenvalues of $M^G_{\pi,\nu}$ in \textnormal{(\ref{eq-M})} and let $f|_B$ be the restriction of $f$ to a subset $B$ of $V$. Suppose $1\le k_0\le n$.
\begin{itemize}
\item[(1)] If $\xi_{\lambda}|_{\{1,...,k_0\}}$ is of type $i$ with $(-1)^i\xi_\lambda(k_0)>0$ and $\widetilde{\xi}_\lambda|_{\{k_0,...,n\}}$ is of type $j$ with $(-1)^j\widetilde{\xi}_\lambda(k_0)>0$, then $\lambda^G_{i+j-2}<\lambda<\lambda^G_{i+j-1}$.

\item[(2)] If $\xi_{\lambda}|_{\{1,...,k_0\}}$ is of type $i$ with $(-1)^i\xi_\lambda(k_0)<0$ and $\widetilde{\xi}_\lambda|_{\{k_0,...,n\}}$ is of type $j$ with $(-1)^j\widetilde{\xi}_\lambda(k_0)<0$, then $\lambda^G_{i+j-1}<\lambda<\lambda^G_{i+j+1}$.

\item[(3)] If $\xi_{\lambda}|_{\{1,...,k_0\}}$ is of type $i$ with $(-1)^i\xi_\lambda(k_0)>0$ and $\widetilde{\xi}_\lambda|_{\{k_0,...,n\}}$ is of type $j$ with $(-1)^j\widetilde{\xi}_\lambda(k_0)<0$, then $\lambda^G_{i+j-2}<\lambda<\lambda^G_{i+j}$.
\end{itemize}
\end{prop}
\begin{proof}
By Proposition \ref{p-xi}, $\xi_\lambda(n)$ is a polynomial of degree $n-1$ satisfying
\[
 (-1)^{i+1}\xi_{\lambda^G_i}(n)>0,\,\, \forall 0\le i<n,\quad (-1)^{i+1}\xi_{\beta_i}(n)>0,\,\, \forall 1\le i<n-1.
\]
This implies that there are $w_i\in(\beta_i,\lambda^G_{i+1})$, $0\le i\le n-2$, such that $(-1)^{i+1}\xi_\lambda(n)>0$ for $\lambda\in(w_{i-1},w_i)$ and $0\le i\le n-1$ with $w_{-1}=-\infty$ and $w_{n-1}=\infty$.

The proofs for (1)-(3) in Proposition \ref{p-twosided} are similar and we deal with (1) only. By the Euler-Lagrange equations in (\ref{eq-EL2}), it is easy to see that, for $1\le l<n$, $\xi_{\lambda^G_l}$ and $\widetilde{\xi}_{\lambda^G_l}$ are eigenvectors of $M^G_{\pi,\nu}$ in (\ref{eq-M}) associated with $\lambda^G_l$, which implies $\xi_{\lambda^G_l}=-\xi_{\lambda^G_l}(n)\widetilde{\xi}_{\lambda^G_l}$. First, assume that $\lambda\le\lambda^G_{i+j-2}$. By Proposition \ref{p-xi}, $\xi_{\lambda^G_{i+j-2}}|_{\{1,...,k_0\}}$ is of type at least $i$ and $\widetilde{\xi}_{\lambda^G_{i+j-2}}|_{\{k_0,...,n\}}$ is of type at least $j$. This implies that
the patching of $\xi_{\lambda^G_{i+j-2}}|_{\{1,...,k_0\}}$ and $-\xi_{\lambda^G_{i+j-2}}(n)\widetilde{\xi}_{\lambda^G_{i+j-2}}|_{\{k_0,...,n\}}$, which equals to $\xi_{\lambda^G_{i+j-2}}$, is of type at least $i+j-1$. This is a contradiction.

Next, assume that $\lambda\ge\lambda^G_{i+j-1}$. By Proposition \ref{p-xi}, we may choose $a_1<\lambda$  (resp. $a_2<\lambda$) such that $\xi_\lambda|_{\{1,...,k_0\}}$ (resp. $\widetilde{\xi}_\lambda|_{\{k_0,...,n\}}$) changes the type at $a_1$ (resp. $a_2$). If $\lambda^G_{i+j-1}\le\min\{a_1,a_2\}$, then a similar reasoning as before implies that $\xi_{\lambda^G_{i+j-1}}$ is of type at most $i+j-2$, a contradiction. If $\min\{a_1,a_2\}<\lambda^G_{i+j-1}<\max\{a_1,a_2\}$, then exactly one of $\xi_{\lambda^G_{i+j-1}}|_{\{1,...,k_0\}}$ and $\widetilde{\xi}_{\lambda^G_{i+j-1}}|_{\{k_0,...,n\}}$ does not change its type. This implies that the gluing point $k_0$ can not be a local extremum and, thus, the patching function is of type at most $i+j-2$, another contradiction! According to the discussion in the first paragraph of this proof, if $\lambda^G_{i+j-1}\ge\max\{a_1,a_2\}$, then none of $\xi_{\lambda^G_{i+j-1}}|_{\{1,...,k_0\}}$ and $\widetilde{\xi}_{\lambda^G_{i+j-1}}|_{\{k_0,...,n\}}$ changes type nor, of course, the sign at $k_0$. Consequently, we obtain
$(-1)^{i+j}\xi_{\lambda^G_{i+j-1}}(k_0)\widetilde{\xi}_{\lambda^G_{i+j-1}}(k_0)>0$,
which contradicts the fact $\xi_{\lambda^G_{i+j-1}}=-\xi_{\lambda^G_{i+j-1}}(n)\widetilde{\xi}_{\lambda^G_{i+j-1}}$.
\end{proof}

\begin{prop}\label{p-signchange}
For $\lambda>0$ and $1\le k\le n-1$, let $s_k(\lambda)$ be the $k$th sign change of $\xi_\lambda$ defined by $s_0:=0$ and $s_{k+1}(\lambda):=\inf\{l>s_k(\lambda)|\xi_\lambda(l)\xi_\lambda(l-1)<0\text{ or } \xi_\lambda(l)=0\}$, where $\inf\emptyset:=n+1$. Then, for $0<\lambda_1<\lambda_2$, $s_k(\lambda_1)\ge s_k(\lambda_2)$ for all $1\le k\le n-1$.
\end{prop}
\begin{proof}
Let $1\le k\le n-1$. If $s_k(\lambda_1)=n+1$, then it is clear that $s_k(\lambda_1)\ge s_k(\lambda_2)$. Suppose that $s_k(\lambda_1)=\ell\le n$. Obviously, $\xi_{\lambda_1}|_{\{1,...,\ell\}}$ is of type $k$. Referring to (\ref{eq-ail}), let $\lambda^\ell_1,...,\lambda^\ell_{\ell-1}$ be the roots of $\det A_{\ell-1}(\lambda)=0$ and $\alpha^\ell_1,...,\alpha^\ell_{\ell-2}$ be roots of $\det A_{\ell-2}(\lambda)=0$. According to the first paragraph of the proof for Proposition \ref{p-twosided}, there are $w^\ell_i\in(\alpha^\ell_{i-1},\lambda^\ell_i)$ with $1\le i\le \ell-1$ such that $(-1)^{i+1}\xi_\lambda(\ell)>0$ for $\lambda\in(w^\ell_i,w^\ell_{i+1})$ and $1\le i\le \ell-1$, where $\alpha^\ell_0:=0$. Since $\xi_{\lambda_1}(\ell)\xi_{\lambda^\ell_k}(\ell)\ge 0$, one has $w^\ell_k\le \lambda_1<\alpha^\ell_k$. As it is assumed that $\lambda_2>\lambda_1$, if $\lambda_2>\alpha^\ell_k$, then $\xi_{\lambda_2}|_{\{1,...,\ell\}}$ is of type at least $k+1$ and, consequently, $s_k(\lambda_2)<\ell=s_k(\lambda_1)$. If $\lambda1<\alpha^\ell_k$, then $\xi_{\lambda_2}|_{\{1,...,\ell\}}$ is type $k$ and $\xi_{\lambda_2}(\ell)<0$. This implies $s_k(\lambda_2)\le \ell=s_k(\lambda_1)$, as desired.
\end{proof}

\subsection{Bounding eigenvalues from below}

Motivated by Theorem \ref{t-Miclo2}, we introduce another scheme generalizing (\ref{eq-psigen}) to bound the other eigenvalues of $M^G_{\pi,\nu}$ from below.
\begin{defn}\label{def-xii}
For $\lambda>0$, let $\xi_\lambda$ be a function in Definition \ref{def-xil}. If $\xi_\lambda$ is of type $i$, $1\le i\le n-1$, with peak-valley points $1=a_1<b_1\le a_2<b_2\le\cdots\le a_i<b_i\le n$, then define
\[
 \xi_\lambda^{(j)}(k)=\begin{cases}\xi_\lambda(k)&\text{for }k\le b_j\\\xi_\lambda(k)=\xi_\lambda(b_j)&\text{for }k>b_j\end{cases},\quad\forall 1\le j<i
\]
and set $\xi_\lambda^{(j)}=\xi_\lambda$ for $i\le j\le n-1$.
\end{defn}

\begin{rem}\label{r-xi2}
For $\lambda>0$, if $\xi_\lambda$ is of type $i$, then $\xi_\lambda^{(j)}$ is of type $j$ for $j<i$. Moreover, for $k<b_j$,
\begin{align}
 \xi_\lambda^{(j)}(k+1)&=\xi^{(j)}_\lambda(k)+\frac{[\xi^{(j)}_\lambda(k)-\xi^{(j)}_\lambda(k-1)]\nu(k-1,k)-\lambda\pi(k)\xi^{(j)}_\lambda(k)}{\nu(k,k+1)}\notag\\
&=\xi^{(j)}_\lambda(k)-\frac{\lambda[\pi(1)\xi^{(j)}_\lambda(1)+\cdots+\pi(k)\xi^{(j)}_\lambda(k)]}{\nu(k,k+1)},\notag
\end{align}
and, for $b_j\le k<n$,
\[
 \xi_\lambda^{(j)}(k+1)=\xi^{(j)}_\lambda(k)+\frac{F_j([\xi^{(j)}_\lambda(k)-\xi^{(j)}_\lambda(k-1)]\nu(k-1,k)-\lambda\pi(k)\xi^{(j)}_\lambda(k))}{\nu(k,k+1)},
\]
where $F_j(t)=\max\{t,0\}$ if $j$ is odd, and $F_j(t)=\min\{t,0\}$ if $j$ is even. Note that $\xi_\lambda^{(1)}$ is exactly $\phi_\lambda$ in Proposition \ref{p-shape}.
\end{rem}

Thereafter, let $\mathcal{L}$ and $\mathcal{L}^{(i)}$ be functions on $(0,\infty)$ defined by
\begin{equation}\label{eq-LLi}
 \mathcal{L}(\lambda)=\frac{\mathcal{E}_\nu(\xi_\lambda,\xi_\lambda)}{\textnormal{Var}_\pi(\xi_\lambda)},
 \quad\mathcal{L}^{(i)}(\lambda)=\frac{\mathcal{E}_\nu(\xi_\lambda^{(i)},\xi_\lambda^{(i)})}
 {\textnormal{Var}_\pi(\xi_\lambda^{(i)})},\quad\forall 1\le i\le n-1,
\end{equation}
where $\xi_\lambda$ and $\xi_\lambda^{(i)}$ are functions in Definitions \ref{def-xil}-\ref{def-xii}.

\begin{rem}\label{r-LLi}
Note that $\mathcal{L}=\mathcal{L}^{(n-1)}$. By a similar reasoning as
in the proof for (\ref{eq-Enu}), one can show that, for $\lambda>0$,
\[
 \mathcal{L}(\lambda)=\lambda+\frac{\lambda\pi(\xi_\lambda)[\pi(\xi_\lambda)-\xi_\lambda(n)]}
 {\textnormal{Var}_\pi(\xi_\lambda)},\quad \mathcal{L}^{(i)}(\lambda)=\lambda+\frac{\lambda\pi(\xi^{(i)}_\lambda)[\pi(\xi^{(i)}_\lambda)
 -\xi^{(i)}_\lambda(n)]}{\textnormal{Var}_\pi(\xi^{(i)}_\lambda)}.
\]
From Proposition \ref{p-xi}, it  followss immediately
that $\mathcal{L}(\lambda)=\mathcal{L}^{(i)}(\lambda)$ for $\lambda\in(0,\alpha_i]$.
\end{rem}

To explore further $\mathcal{L}$ and $\mathcal{L}^{(i)}$, we need more information of $\pi(\xi_\lambda)$, $\pi(\xi_\lambda^{(i)})$, $\pi(\xi_\lambda)-\xi_\lambda(n)$ and $\pi(\xi_\lambda^{(i)})-\xi_\lambda^{(i)}(n)$.

\begin{lem}\label{l-xi2}
Let $\xi_\lambda$ be the function in {\em Definition \ref{def-xil}}
and $\lambda^G_i,\alpha_i$ be constants in {\em Proposition \ref{p-xi}}. Then, $\pi(\xi_\lambda)-\xi_\lambda(n)=0$ has $n-1$ distinct roots, say $\beta_0<\beta_1<\cdots<\beta_{n-2}$, which satisfy $\beta_0=0$ and $\alpha_i<\beta_i<\lambda^G_{i+1}$ for $1\le i\le n-2$. Furthermore, $\pi(\xi_\lambda)-\xi_\lambda(n)>0$ for $\lambda\in(\beta_{2i-1},\beta_{2i})$ and $\pi(\xi_\lambda)-\xi_\lambda(n)<0$ for $\lambda\in(\beta_{2i},\beta_{2i+1})$, with $\beta_{-1}=-\infty$ and $\beta_{n-1}=\infty$.
\end{lem}
\begin{proof}
Set $u(\lambda):=\pi(\xi_\lambda)-\xi_\lambda(n)$. According to Definition \ref{def-xil}, $u(\lambda)$ is a polynomial of degree $n-1$ and satisfies $u(0)=0$. Note that $\pi(\xi_\lambda)=0$ for $\lambda\in\{\lambda^G_1,...,\lambda^G_{n-1}\}$. If $i$ is odd, then $\xi_{\lambda^G_i}(n-1)<\xi_{\lambda^G_i}(n)$. This implies $\xi_{\lambda^G_i}(n)>0$ and, hence, $u(\lambda^G_i)<0$. Similarly, if $i$ is even, then $u(\lambda^G_i)>0$.

By Lemma \ref{l-xi} and Proposition \ref{p-xi}, if $\lambda=\alpha_i$ with odd $i$, then $\xi_{\alpha_i}$ is of type $i$ with $\xi_{\alpha_i}(n-1)=\xi_{\alpha_i}(n)$. This implies $\xi_{\alpha_i}(n)>0$ and $\pi(\xi_{\alpha_i})=\pi(n)\xi_{\alpha_i}(n)$, which yields $u(\alpha_i)<0$. Similarly, one can show that $u(\alpha_i)>0$ if $i$ is even.
\end{proof}

\begin{rem}\label{r-xii}
We consider the sign of $\pi(\xi_\lambda^{(i)})$ and $\pi(\xi_\lambda^{(i)})-\xi_\lambda^{(i)}(n)$ in this remark. By Proposition \ref{p-xi}, $\xi_\lambda^{(i)}=\xi_\lambda$ for $\lambda\le\alpha_i$. If $\lambda>\alpha_i$ with $1\le i\le n-2$, then $\xi_\lambda$ is of type $j$ with $j>i$. Fix $1\le i\le n-2$ and set $k_0=k_0(\lambda)=\min\{k|\xi_\lambda^{(i)}(j)=\xi_\lambda^{(i)}(n),\,\forall k\le j\le n\}$. Clearly, $k_0(\lambda)\le n-1$ for $\lambda>\alpha_i$. Observe that, for $\lambda>\alpha_i$ with odd $i$, $\xi_\lambda(k_0-1)<\xi_\lambda(k_0)\ge\xi_\lambda(k_0+1)$, which implies $\sum_{k=1}^{k_0-1}\pi(k)\xi_\lambda(k)<0$ and $\sum_{k=1}^{k_0}\pi(k)\xi_\lambda(k)\ge 0$. A similar reasoning for the case of even $i$ gives $\sum_{k=1}^{k_0-1}\pi(k)\xi_\lambda(k)>0$ and $\sum_{k=1}^{k_0}\pi(k)\xi_\lambda(k)\le 0$. Consequently, we obtain
\begin{equation}\label{eq-xii}
 (-1)^{i-1}\pi(\xi_\lambda^{(i)})>0,\quad (-1)^i[\pi(\xi_\lambda^{(i)})-\xi_\lambda^{(i)}(n)]>0,
\end{equation}
for $\lambda>\alpha_i$ and $1\le i\le n-2$. Note that, by Proposition \ref{p-xi}, $\xi_\lambda^{(i)}=\xi_\lambda$ for $\lambda\le\alpha_i$. In addition with Remark \ref{r-xil}, Lemma \ref{l-xi2} and the continuity of $\xi_\lambda^{(i)}$, the first inequality of (\ref{eq-xii}) holds for $\lambda>\lambda_i^G$ and the second inequalities of (\ref{eq-xii}) hold for $\lambda>\beta_{i-1}$.
\end{rem}

According to Lemma \ref{l-xi2} and Remark \ref{r-xii}, we derive a generalized version of Proposition \ref{p-main1} in the following.

\begin{prop}\label{p-LLi}
Let $n\ge 3$ and $1\le i\le n-1$. For $\lambda>0$, let $\xi_\lambda,\xi_\lambda^{(i)}$ be the functions in {\em Definition \ref{def-xil}} and $\beta_i$ be the constants in {\em Lemma \ref{l-xi2}}.
\begin{itemize}
\item[(1)] For $\lambda>\beta_{i-1}$, the following are equivalent.
\begin{itemize}
\item[(1-1)] $\mathcal{E}_\nu(\xi_\lambda^{(i)},\xi_\lambda^{(i)})
    =\lambda\textnormal{Var}_\pi(\xi_\lambda^{(i)})$.

\item[(1-2)] $\pi(\xi_\lambda^{(i)})=0$.

\item[(1-3)] $\lambda=\lambda^G_i$.
\end{itemize}

\item[(2)] For $\beta_{i-1}<\lambda<\beta_i$, the following are equivalent.
\begin{itemize}
\item[(2-1)] $\mathcal{E}_\nu(\xi_\lambda,\xi_\lambda)=\lambda\textnormal{Var}_\pi(\xi_\lambda)$.

\item[(2-2)] $\pi(\xi_\lambda)=0$.

\item[(2-3)] $\lambda=\lambda^G_i$.
\end{itemize}
\end{itemize}
\end{prop}
\begin{proof}
The proof for Proposition \ref{p-LLi} (2) is similar to the proof for Proposition \ref{p-LLi} (1) and we deal only with the latter. By Lemma \ref{l-xi2} and Remark \ref{r-xii}, one has
\[
 \pi(\xi_\lambda^{(i)})[\pi(\xi_\lambda^{(i)})-\xi_\lambda^{(i)}(n)]\begin{cases}<0&\text{for } \lambda>\lambda^G_i\\>0&\text{for }\beta_{i-1}<\lambda<\lambda^G_i\end{cases}.
\]
This proves the equivalence of (1-1) and (1-2). Under the assumption of (1-2)
and  using Remark \ref{r-xil}, one has $\lambda\le\alpha_i$.
This implies $\xi_\lambda^{(i)}=\xi_\lambda$ is an eigenvector for $M^G_{\pi,\nu}$ with associated eigenvalue $\lambda$. As $\lambda\in(\beta_{i-1},\alpha_i]$, it must be the case $\lambda=\lambda^G_i$. This gives (1-3), while (1-3)$\Ra$(1-2) is obvious and omitted.
\end{proof}

\begin{rem}
It is worthwhile to note that if (1-1) and (2-1) of Proposition \ref{p-LLi} are removed, then the equivalence in (1) holds for $\lambda>\lambda^G_{i-1}$ and the equivalence in (2) holds for $\lambda\in(\lambda^G_{i-1},\lambda^G_{i+1})$. Once $\lambda^G_{i-1}$ is known, we can determine $\lambda^G_i$ using the sign of $\pi(\xi_\lambda^{(i)})$. See Theorem \ref{t-xi} for details.
\end{rem}

\begin{rem}
Note that condition (4) of Proposition \ref{p-main1} is not included in Proposition \ref{p-LLi}. In fact, the equivalence may fail, that is, there may
exist some $\lambda\in(\beta_{i-1},\beta_i)\setminus\{\lambda^G_i\}$ such that $\mathcal{E}_\nu(\xi_\lambda,\xi_\lambda)/\textnormal{Var}_\pi(\xi_\lambda)=\lambda^G_i$. See Example \ref{ex-Ehrenfest} for a counterexample.
\end{rem}

As Proposition \ref{p-LLi} focuses on the characterization of zeros of $\mathcal{L}(\lambda)-\lambda$, the following theorem concerns the sign of $\mathcal{L}(\lambda)-\lambda$.
\begin{thm}\label{t-xi}
Let $\lambda^G_i,\alpha_i,\beta_i$ be the constants in
{\em Proposition \ref{p-xi}} and {\em Lemma \ref{l-xi2}},
and $\mathcal{L}$ be the function in \textnormal{(\ref{eq-LLi})}. Then, $\lambda^G_1,...,\lambda^G_{n-1},\beta_1,...,\beta_{n-2}$ are fixed points of $\mathcal{L}$ and, for $1\le i\le n-2$,
\begin{itemize}
\item[(1)] $\mathcal{L}(\lambda)<\lambda$ for $\lambda\in(\lambda^G_i,\beta_i)$.

\item[(2)] $\mathcal{L}(\lambda)>\lambda$ for $\lambda\in(\beta_i,\lambda^G_{i+1})$.

\item[(3)] $\mathcal{L}^{(i)}(\lambda)<\lambda$ for $\lambda\in(\lambda^G_i,\infty)$.
\end{itemize}
\end{thm}
\begin{proof}
Immediate from Lemma \ref{l-xi2} and Remarks \ref{r-LLi}-\ref{r-xii}.
\end{proof}

By Theorem \ref{t-xi}, we obtain a lower bound on any specified eigenvalue of $M^G_{\pi,\nu}$.
\begin{cor}\label{c-LLi}
Let $1\le i\le n-1$ and $\lambda_0>\lambda^G_i$. Consider the sequence $\lambda_{\ell+1}=\mathcal{L}^{(i)}(\lambda_\ell)$ with $\ell\ge 0$ and set
\[
 \lambda^*=\begin{cases}\lim_{\ell\ra\infty}\lambda_\ell&\text{if $\lambda_\ell$  converges}\\\sup_{\ell\in I}\lambda_\ell&\text{if $\lambda_\ell$ diverges}\end{cases},
\]
where $I=\{\ell|\lambda_{\ell-1}>\lambda_\ell<\lambda_{\ell+1}\}$. Then, $\lambda^*\le\lambda^G_i$.
\end{cor}
It is not clear yet whether the sequence $\lambda_\ell$ in Corollary \ref{c-LLi} is convergent, even locally. This subject will be discussed in the next subsection. Now, we establish some relations between the roots of $\det A_i(\lambda)=0$ and the shape of $\xi_\lambda^{(i)}$. This is a generalization of Proposition \ref{p-shape}.

\begin{prop}\label{p-shape2}
For $1\le i\le n-1$, let $A_i(\lambda)$ be the matrix in \textnormal{(\ref{eq-ail})}, $\theta^{(i)}_1<\cdots<\theta^{(i)}_i$ be zeros of $\det A_i(\lambda)=0$ and set $\theta_i^{(i-1)}:=\infty$. Referring to the notation in
{\em Proposition \ref{p-xi}}, it holds true that, for $1\le i\le n-1$, \begin{itemize}
\item[(1)] $\lambda^G_i=\theta^{(n-1)}_i<\alpha_i=\theta^{(n-2)}_i<\cdots<\theta^{(i)}_i$.

\item[(2)] $\xi_\lambda^{(i)}(j)\ne\xi_\lambda^{(i)}(j+1)=\cdots=\xi_\lambda^{(i)}(n)$ for  $\lambda\in[\theta^{(j)}_i,\theta^{(j-1)}_i)$ and $i\le j\le n-2$.

\item[(3)] $\xi_\lambda^{(i)}(n-1)\ne\xi_\lambda^{(i)}(n)$ for $\lambda\in(\theta_{i-1}^{(n-2)},\theta_i^{(n-2)})$ and $i\le n-1$.
\end{itemize}
\end{prop}
\begin{proof}
The order in (1) is a simple application of Lemma \ref{l-mat}. For (2), fix $1\le i\le n-1$ and set $\gamma(\lambda)=\min\{j|\xi_\lambda^{(i)}(k)=\xi_\lambda^{(i)}(n),\,\forall j\le k\le n\}$ and $B(\lambda)=\{1,2,...,\gamma(\lambda)\}$, $B^+(\lambda)=B(\lambda)\cup\{\gamma(\lambda)+1\}$. Clearly, $i+1\le\gamma(\lambda)\le n$. We use the notation $\xi_\lambda|_C$ to denote the restriction of $\xi_\lambda$ to a set $C$. Suppose that $i$ is odd. By Remark \ref{r-xi2}, $\xi_\lambda^{(i)}=\xi_\lambda$ on $B(\lambda)$ and $\xi_\lambda|_{B(\lambda)}$ is of type $i$ with
\[
 \xi_\lambda(\gamma(\lambda)-1)<\xi_\lambda(\gamma(\lambda))\ge\xi_\lambda(\gamma(\lambda)+1).
\]
By Lemma \ref{l-xi}(1), if $\xi_\lambda(\gamma(\lambda)+1)<\xi_\lambda(\gamma(\lambda))$, then there is $\epsilon>0$ such that, for $|\delta|<\epsilon$, $\xi_{\lambda+\delta}|_{B(\lambda)}$ is of type $i$ and
\[
 \xi_{\lambda+\delta}(\gamma(\lambda)-1)<\xi_{\lambda+\delta}(\gamma(\lambda))
>\xi_{\lambda+\delta}(\gamma(\lambda)+1).
\]
This implies $\gamma(\lambda+\delta)=\gamma(\lambda)$ for $\delta\in(-\epsilon,\epsilon)$.
By Lemma \ref{l-xi}(2), if $\xi_\lambda(\gamma(\lambda)+1)=\xi_\lambda(\gamma(\lambda))$, then there is $\epsilon>0$ such that, for $\delta\in(-\epsilon,0)$, $\xi_{\lambda+\delta}|_{B^+(\lambda)}$ is of type $i$ with
\[
 \xi_{\lambda+\delta}(\gamma(\lambda)-1)<\xi_{\lambda+\delta}(\gamma(\lambda))
<\xi_{\lambda+\delta}(\gamma(\lambda)+1),
\]
and, for $\delta\in(0,\epsilon)$, $\xi_{\lambda+\delta}|_{B^+(\lambda)}$ is of type $i+1$ with
\[
 \xi_{\lambda+\delta}(\gamma(\lambda)-1)<\xi_{\lambda+\delta}(\gamma(\lambda))
>\xi_{\lambda+\delta}(\gamma(\lambda)+1).
\]
This yields $\gamma(\lambda+\delta)=\gamma(\lambda)$ for $\delta\in(0,\epsilon)$ and $\gamma(\lambda+\delta)=\gamma(\lambda)+1$ for $\delta\in(-\epsilon,0)$. The proof for the case of even $i$ is similar and we conclude from the above that $\gamma(\lambda)$ is a non-increasing and right-continuous function taking values on $\{i+1,...,n\}$. Let $c_{i+1}>\cdots>c_{n-1}$ be the discontinuous points of $\gamma(\lambda)$ such that $\gamma(c_j)=j$ for $i+1\le j\le n-1$. As a consequence of the above discussion, $\xi_{c_j}|_{\{1,...,j\}}$ is of type $i$ with $\xi_{c_j}(j)=\xi_{c_j}(j+1)$ and this implies $\sum_{k=1}^j\pi(k)\xi_{c_j}(k)=0$. That means $c_j$ is a root of $\det A_{j-1}(\lambda)=0$ for $j=i+1,...,n-1$. By Proposition \ref{p-xi} and the second equality in (1), $\gamma(\lambda)=n$ for $\theta^{(n-2)}_{i-1}<\lambda<\theta^{(n-2)}_i$ and, thus, $c_j\ge \theta^{(n-2)}_i$ for $j\ge i+1$. As a consequence of the interlacing relationship $\theta^{(\ell)}_i<\theta^{(\ell-1)}_i<\theta^{(\ell)}_{i+1}$, it must be $c_j=\theta_i^{(j+1)}$ for $i+1\le j\le n-1$. This finishes the proof.
\end{proof}
\begin{rem}\label{r-shape}
For $1\le i\le n-1$, $\theta^{(i)}_1,...,\theta^{(i)}_i$ are also non-zero eigenvalues of the $(i+1)\times(i+1)$ principal submatrix of (\ref{eq-M}) indexed by $1,...,i+1$.
\end{rem}

\begin{rem}\label{r-shape2}
In fact, by Proposition \ref{p-shape}, $\xi_\lambda^{(1)}(n-1)\ne\xi_\lambda^{(1)}(n)$ for $\lambda\in(0,\theta_1^{(n-2)})$, which is better than Proposition \ref{p-shape2}(3).
\end{rem}

\subsection{Local convergence of $\mathcal{L}$}
This subsection is dedicated to the local convergence of $\mathcal{L}$ in (\ref{eq-LLi}). Let $\alpha_i,\beta_i,\lambda^G_i$ be the constants in Proposition \ref{p-xi} and Lemma \ref{l-xi2}. As before, let $\zeta_0=\mathbf{1},...,\zeta_{n-1}$ denote the $L^2(\pi)$-normalized eigenvectors of $M^G_{\pi,\nu}$ associated with $\lambda^G_0,...,\lambda^G_{n-1}$. Clearly, $\xi_{\lambda^G_i}=-\zeta_i/\zeta_i(1)$ and $\xi_\lambda=\sum_{i=0}^{n-1}\rho_i(\lambda)\zeta_i$, where $\rho_i(\lambda)=\pi(\xi_\lambda\zeta_i)$ for $0\le i\le n-1$. Note that $\rho_i(\lambda)$ is a polynomial of degree $n-1$ and satisfies $\rho_i(\lambda_j)=-\delta_i(j)/\zeta_i(1)$ for $i,j\in\{0,1,...,n-1\}$. This implies
\begin{equation}\label{eq-rhoi1}
 \rho_0(\lambda)=-\prod_{j=1}^{n-1}\frac{\lambda^G_j-\lambda}{\lambda^G_j},\quad
\rho_i(\lambda)=-\frac{\lambda}{\zeta_i(1)\lambda^G_i}\prod_{j=1,j\ne i}^{n-1}\frac{\lambda^G_j-\lambda}{\lambda^G_j-\lambda^G_i},
\end{equation}
for all $1\le i\le n-1$. Moreover, by multiplying (\ref{eq-EL2}) with $\xi_\lambda(k)$ and summing up $k$, we obtain $\mathcal{E}_\nu(\xi_\lambda,\zeta_i)=\lambda^G_i\rho_i(\lambda)$. In the same spirit, one can show that $\mathcal{E}_\nu(\xi_\lambda,\zeta_i)=\lambda[\rho_i(\lambda)-\zeta_i(n)\rho_0(\lambda)]$ using Definition \ref{def-xil}. Putting both equations together yields
\begin{equation}\label{eq-rhoi2}
 \rho_i(\lambda)=\frac{\lambda\zeta_i(n)}{\lambda-\lambda^G_i}\rho_0(\lambda),\quad\forall 0\le i\le n-1.
\end{equation}
As a consequence of Remark \ref{r-LLi}, this gives
\begin{equation}\label{eq-Lsp}
 \mathcal{L}(\lambda)=\frac{\sum_{i=1}^{n-1}\lambda^G_i\rho_i^2(\lambda)}{\sum_{i=1}^{n-1}\rho_i^2(\lambda)}=\lambda+\frac{\sum_{i=1}^{n-1}(\lambda^G_i-\lambda)^{-1}\zeta^2_i(n)}
{\sum_{i=1}^{n-1}(\lambda^G_i-\lambda)^{-2}\zeta^2_i(n)},
\end{equation}
for $\lambda\notin\{\lambda^G_0,...,\lambda^G_{n-1}\}$. The next proposition follows immediately from the second equation in (\ref{eq-rhoi1}) and (\ref{eq-rhoi2}).
\begin{prop}\label{p-zeta}
Let $\lambda^G_1,...,\lambda^G_{n-1}$ be the non-zero eigenvalues of $M^G_{\pi,\nu}$ in \textnormal{(\ref{eq-M})} and $\zeta_1,...,\zeta_{n-1}$ be the corresponding $L^2(\pi)$-normalized eigenvectors. Then,
\[
 \zeta_i(1)\zeta_i(n)=-\prod_{j=1,j\ne i}^{n-1}\frac{\lambda^G_j}{\lambda^G_j-\lambda^G_i},\quad\forall 1\le i\le n-1.
\]
\end{prop}

Set $u(\lambda)=\sum_{j=1}^{n-1}(\lambda^G_j-\lambda)^{-1}\zeta_j^2(n)$. By Theorem \ref{t-xi}, $\beta_1,...,\beta_{n-2}$ are zeros of $u(\lambda)\prod_{j=1}^{n-1}(\lambda^G_j-\lambda)$, which is a polynomial of degree $n-2$. This implies
\[
 u(\lambda)=C\left(\prod_{j=1}^{n-1}\frac{1}{\lambda^G_j-\lambda}\right)\left(\prod_{j=1}^{n-2}(\beta_j-\lambda)\right),
\]
where $C=\frac{\lambda_1\cdots\lambda_{n-1}}{\beta_1\cdots\beta_{n-2}}\sum_{j=1}^{n-1}\zeta_j^2(n)/\lambda^G_j$. Putting this back to $\mathcal{L}$ yields
\begin{equation}\label{eq-Lsp2}
 \frac{1}{\mathcal{L}(\lambda)-\lambda}=\frac{u'(\lambda)}{u(\lambda)}=\sum_{j=1}^{n-1}\frac{1}{\lambda^G_j-\lambda}-\sum_{j=1}^{n-2}\frac{1}{\beta_j-\lambda},
\end{equation}
for $\lambda\notin\{\lambda^G_0,...,\lambda^G_{n-1},\beta_1,...,\beta_{n-2}\}$.

\begin{prop}\label{p-L}
Let $\mathcal{L}$ be the function in \textnormal{(\ref{eq-LLi})}, $\lambda^G_i$ be the eigenvalue of $M^G_{\pi,\nu}$ and $\beta_i$ be the constant in
{\em Lemma \ref{l-xi2}}. Let $D_i=\sum_{j=1}^{n-2}(\beta_j-\lambda^G_i)^{-1}-\sum_{j=1,j\ne i}^{n-1}(\lambda^G_j-\lambda^G_i)^{-1}$ with $1\le i\le n-1$. Then, for $2\le i\le n-2$,
\begin{itemize}
\item[(1)] If $D_i<0$, then there is $\tau\in(\lambda^G_i,\beta_i)$ such that $\mathcal{L}$ is strictly increasing on $(\beta_{i-1},\lambda^G_i)\cup(\tau,\beta_i)$ and strictly decreasing on $(\lambda^G_i,\tau)$.

\item[(2)] If $D_i>0$, then there is $\eta\in(\beta_{i-1},\lambda^G_i)$ such that $\mathcal{L}$ is strictly increasing on $(\beta_{i-1},\eta)\cup(\lambda^G_i,\beta_i)$ and strictly increasing on $(\eta,\lambda^G_i)$.

\item[(3)] If $D_i=0$, then $\mathcal{L}$ is strictly increasing on $(\beta_{i-1},\beta_i)$.
\end{itemize}
\end{prop}
\begin{proof}
Using (\ref{eq-Lsp}) and (\ref{eq-Lsp2}), one can show that $\mathcal{L}'(\lambda^G_i)=0$ and
\begin{equation}\label{eq-L''}
 \mathcal{L}''(\lambda^G_i)=\sum_{j=1,j\ne i}^{n-1}\frac{\zeta_i^2(n)}{\lambda^G_j-\lambda^G_i}=2\left[\sum_{j=1}^{n-2}\frac{1}{\beta_j-\lambda^G_i}-\sum_{j=1,j\ne i}^{n-1}\frac{1}{\lambda^G_j-\lambda^G_i}\right]=2D_i.
\end{equation}
To prove (1) and (2), it suffices to show that if $\mathcal{L}'(\tau)=0$ for some $\tau\in(\lambda^G_i,\beta_i)$, then $\tau$ is a local minimum of $\mathcal{L}$, and if $\mathcal{L}'(\eta)=0$ for some $\eta\in(\beta_{i-1},\lambda^G_i)$, then $\eta$ is a local maximum of $\mathcal{L}$. We discuss the first case, whereas the second case is similar and is omitted. Recall that $u(\lambda)=\sum_{j=1}^{n-1}(\lambda_j^G-\lambda)^{-1}\zeta_j^2(n)$. As $\tau$ is a critical point for $\mathcal{L}$, one has $2(u'(\tau))^2=u(\tau)u''(\tau)$. This implies
\[
 \mathcal{L}''(\tau)=\frac{u(\tau)[3(u''(\tau))^2-2u'(\tau)u'''(\tau)]}{2(u'(\tau))^3}>0,
\]
where the last inequality uses the fact that $u(\lambda)<0$, for $\lambda\in(\lambda^G_i,\beta_i)$, and
\[
 3(u''(\lambda))^2-2u'(\lambda)u'''(\lambda)=-12\sum_{1\le i<j\le n-1}\left[\frac{(\lambda^G_i-\lambda^G_j)\zeta_i(n)\zeta_j(n)}{(\lambda^G_i-\lambda)^2(\lambda^G_j-\lambda)^2}\right]^2<0.
\]
This proves (1) and (2).

To see (3), we assume that $D_i=0$. Computations show that
\begin{align}
 &\frac{\mathcal{L}(\lambda)-\lambda^G_i}{\mathcal{L}(\lambda)-\lambda}
=(\lambda-\lambda^G_i)\left[\sum_{j=1,j\ne i}^{n-1}\frac{1}{\lambda^G_j-\lambda}-\sum_{j=1}^{n-2}\frac{1}{\beta_j-\lambda}\right]\notag\\
=&(\lambda-\lambda^G_i)^2\left[\sum_{j=1,j\ne i}^{n-1}\frac{1}{(\lambda^G_j-\lambda)(\lambda^G_j-\lambda^G_i)}-\sum_{j=1}^{n-1}\frac{1}{(\beta_j-\lambda)(\beta_j-\lambda^G_i)}\right]<0,\notag
\end{align}
for $\lambda\in(\beta_{i-1},\lambda^G_i)\cup(\lambda^G_i,\beta_i)$, where the last inequality uses the fact that $(\lambda^G_j-\lambda)(\lambda^G_j-\lambda^G_i)>(\beta_j-\lambda)(\beta_j-\lambda^G_i)$ for $j<i$ and $(\lambda^G_j-\lambda)(\lambda^G_j-\lambda^G_i)>(\beta_{j-1}-\lambda)(\beta_{j-1}-\lambda^G_i)$
for $j>i$. By Theorem \ref{t-xi}, this implies $\mathcal{L}(\lambda)>\lambda^G_i$ for $\lambda\in(\lambda^G_i,\beta_i)$ and $\mathcal{L}(\lambda)<\lambda^G_i$ for $\lambda\in(\beta_{i-1},\lambda^G_i)$. The desired property comes immediate from the discussion in the previous paragraph.
\end{proof}

\begin{rem}
Note that $D_1>0$ and $D_{n-1}<0$. Using the same proof as above, this implies that $\mathcal{L}(\lambda)$ is strictly increasing on $(\lambda^G_1,\beta_1)\cup(\beta_{n-2},\lambda^G_{n-1})$. Moreover, by (\ref{eq-Lsp}), one may compute
\[
 (u'(\lambda))^2\mathcal{L}'(\lambda)=-2\sum_{i<j}\frac{(\lambda^G_i-\lambda^G_j)^2}
{(\lambda^G_i-\lambda)^3(\lambda^G_j-\lambda)^3}<0,\quad\forall \lambda\in(0,\lambda^G_1)\cup(\lambda^G_{n-1},\infty).
\]
This implies $\mathcal{L}(\lambda)$ is strictly decreasing on $(0,\lambda^G_1)\cup(\lambda^G_{n-1},\infty)$ and
\[
 \lim_{\lambda\ra 0}\mathcal{L}(\lambda)=\frac{\sum_{i=1}^{n-1}\zeta_i^2(n)/\lambda^G_i}{\sum_{i=1}^{n-1}\zeta_i^2(n)/(\lambda^G_i)^2},\quad
\lim_{\lambda\ra\infty}\mathcal{L}(\lambda)=\left(\frac{1}{\pi(n)}-1\right)\sum_{i=1}^{n-1}\lambda^G_i\zeta_i^2(n).
\]
\end{rem}

The following local convergence is a simple corollary of Theorem \ref{t-xi} and Proposition \ref{p-L}.
\begin{thm}[Local convergence]\label{t-local}
Let $\lambda_0>0$ and set $\lambda_{\ell+1}=\mathcal{L}(\lambda_\ell)$ for $\ell\ge 0$. Then, there is $\epsilon>0$ such that the sequence $(\lambda_\ell)_{\ell=1}^\infty$ is monotonic and converges to $\lambda^G_i$ for $\lambda_0\in(\lambda^G_i-\epsilon,\lambda^G_i+\epsilon)$ and $1\le i\le n-1$.
\end{thm}

We use the following examples to illustrate the different cases in Proposition \ref{p-L}.

\begin{ex}[Simple random walks]\label{ex-sym}
Let $n>1$. A simple random walk on $\{1,2,...,n\}$ with reflecting probability $1/2$ at the boundary is a birth and death chain with transition matrix given by $K(i,j)=K(1,1)=K(n,n)=1/2$ for $|i-j|=1$. It is easy to see that the uniform probability is the stationary distribution of $K$. In the setting of graph, we have $\nu(i,i+1)=1/(2n)$ and $\pi(i)=1/n$. One may apply the method in \cite{F68} to obtain the following spectral information.
\[
\lambda^G_j=1-\cos\frac{j\pi}{n},\quad \zeta_j(k)=\frac{1}{\sqrt{\lambda^G_j}}\left(\sin\frac{jk\pi}{n}-\sin\frac{j(k-1)\pi}{n}\right),\quad\forall 1\le j<n.
\]
See, e.g.,  \cite[Section 7]{CSal10}. By (\ref{eq-L''}), we get
\[
D_i=\frac{1}{2}\sum_{j=1,j\ne i}^{n-1}\frac{\sin^2(j\pi/n)}{\lambda^G_j(\lambda^G_j-\lambda^G_i)}=\sum_{j=1,j\ne i}^{n-1}\frac{1+\cos(j\pi/n)}{\cos(i\pi/n)-\cos(j\pi/n)}.
\]
Clearly, $D_1>0$ and $D_{n-1}<0$. If $n$ is even, then $D_{n/2}<0$.

\end{ex}

\begin{ex}[Ehrenfest chains]\label{ex-Ehrenfest}
An Ehrenfest chain on $V=\{0,1,...,n\}$ is a Markov chain with transition matrix $K$ given by $K(i,i+1)=1-i/n$ and $K(i+1,i)=(i+1)/n$ for $i=0,...,n-1$. The associated stationary distribution is the unbiased binomial distribution on $V$, that is, $\pi(i)=\binom{n}{i}2^{-n}$ for $i\in V$. To the Ehrenfest chain, the measure $\nu$ is defined by $\nu(i,i+1)=\binom{n-1}{i}2^{-n}$ for $i=0,...,n-1$. Using the group representation for the binary group $\{0,1\}^n$, one may compute
\[
 \lambda_j=\frac{2j}{n},\quad \zeta_j(k)=\binom{n}{j}^{-1/2}\sum_{\ell=0}^j(-1)^\ell\binom{k}{\ell}\binom{n-k}{j-\ell},\quad\forall 1\le j\le n.
\]
Plugging this back into (\ref{eq-L''}) yields
\[
 D_i=\frac{n}{4}\sum_{j=1,j\ne i}^n\frac{\binom{n}{j}}{j-i}\begin{cases}>0&\text{for }i<n/2\\=0&\text{for }i=n/2\\<0&\text{for }i>n/2.\end{cases}.
\]
This example points out the possibility of different signs in $\{D_i|i=1,...,n-1\}$ including $0$.
\end{ex}

\subsection{A remark on the separation for birth and death chains}
In this subsection, we give a new proof of a result, Theorem \ref{t-sep}, which deals with convergence in separation distance for birth and death chains. Let $(X_m)_{m=0}^\infty$ be a birth and death chain with transition matrix $K$ given by (\ref{eq-bdc}). In the continuous time setting, we consider the process $Y_t=X_{N_t}$, where $N_t$ is a Poisson process with parameter $1$ independent of $X_m$. Given the initial distribution $\mu$, which is the distribution of $X_0$, the distributions of $X_m$ and $Y_t$ are respectively $\mu K^m$ and $\mu e^{-t(I-K)}$, where $e^A:=\sum_{l=0}^\infty A^l/l!$. Briefly, we write $H_t=e^{-t(I-K)}$. It is well-known that if $K$ is irreducible, then $\mu H_t$ converges to $\pi$ as $t\ra\infty$. If $K$ is irreducible and $r_i>0$ for some $i$, then $\mu K^m$ converges to $\pi$ as $m\ra\infty$. Concerning the convergence, we consider the separations of $X_m,Y_t$ with respect to $\pi$, which are defined by
\[
 d_{\text{sep}}(\mu,m)=\max_{0\le x\le n}\left\{1-\frac{\mu K^m(x)}{\pi(x)}\right\},\quad
d_{\text{sep}}^c(\mu,t)=\max_{0\le x\le n}\left\{1-\frac{\mu H_t(x)}{\pi(x)}\right\}.
\]
The following theorem is from \cite{DS06}.

\begin{thm}\label{t-sep}
Let $K$ be an irreducible birth and death chain on $\{0,1,...,n\}$ with eigenvalues $\lambda_0=0<\lambda_1<\cdots<\lambda_n$.
\begin{itemize}
\item[(1)] For the discrete time chain, if $p_i+q_{i+1}\le 1$ for all $0\le i<n$, then
\[
 d_{\textnormal{sep}}(0,m)=d_{\textnormal{sep}}(n,m)=\sum_{j=1}^n\left(\prod_{i=1,i\ne j}^n\frac{\lambda_i}{\lambda_i-\lambda_j}\right)(1-\lambda_j)^m.
\]

\item[(2)] For the continuous time chain, it holds true that
\[
 d_{\textnormal{sep}}^c(0,t)=d_{\textnormal{sep}}^c(n,t)=\sum_{j=1}^n\left(\prod_{i=1,i\ne j}^n\frac{\lambda_i}{\lambda_i-\lambda_j}\right)e^{-\lambda_jt}.
\]
\end{itemize}
\end{thm}

Diaconis and Fill \cite{DF90,F92} introduce the concept of dual chain to express the separations in Theorem \ref{t-sep} as the probability of the first passage time. Brown and Shao \cite{BS87} characterize the first passage time using the eigenvalues of $K$ for a special class of continuous time Markov chains including birth and death chains. The idea in \cite{BS87} is also applicable for discrete time chains and this leads to the formula above. See \cite{DS06} for further discussions. Here, we use Proposition \ref{p-zeta} and Lemma \ref{l-bdc} to prove this result directly.

\begin{lem}\label{l-bdc}
Let $K$ be the transition matrix in \textnormal{(\ref{eq-bdc})} with stationary distribution $\pi$. Suppose that $\mu$ is a probability distribution satisfying $\mu(i)/\pi(i)\le \mu(i+1)/\pi(i+1)$ for all $0\le i\le n-1$.
\begin{itemize}
\item[(1)] For the discrete time chain, if $p_i+q_{i+1}\le 1$ for all $0\le i<n$, then $\mu K^m(i)/\pi(i)\le \mu K^m(i+1)/\pi(i+1)$ for all $0\le i<n$ and $m\ge 0$.

\item[(2)] For the continuous time chain, $\mu H_t(i)/\pi(i)\le \mu H_t(i+1)/\pi(i+1)$ for all $0\le i<n$ and $t\ge 0$.
\end{itemize}
\end{lem}
\begin{proof}
Note that (2) follows from (1) if we write $H_t=\exp\{-2t(I-\frac{I+K}{2})\}$. For the proof of (1), observe that
\[
 \frac{\mu K^{m+1}(i)}{\pi(i)}=\frac{\mu K^m(i-1)}{\pi(i-1)}q_i+\frac{\mu K^m(i)}{\pi(i)}r_i+\frac{\mu K^m(i+1)}{\pi(i+1)}p_i,\quad\forall i.
\]
By induction, if $\mu K^m(i)/\pi(i)\le \mu K^m(i+1)/\pi(i+1)$ for $0\le i<n$, then
\begin{align}
 \frac{\mu K^{m+1}(i+1)}{\pi(i+1)}&=\frac{\mu K^m(i)}{\pi(i)}q_{i+1}+\frac{\mu K^m(i+1)}{\pi(i+1)}r_{i+1}+\frac{\mu K^m(i+2)}{\pi(i+2)}p_{i+1}\notag\\
&\ge\frac{\mu K^m(i)}{\pi(i)}q_{i+1}+\frac{\mu K^m(i+1)}{\pi(i+1)}(1-q_{i+1})\notag\\
&\ge\frac{\mu K^m(i)}{\pi(i)}(1-p_i)+\frac{\mu K^m(i+1)}{\pi(i+1)}p_i\ge\frac{\mu K^{m+1}(i)}{\pi(i)}.\notag
\end{align}
\end{proof}

\begin{rem}
Lemma \ref{l-bdc} is also developed in \cite{DLP10} in which it is shown that, for any non-negative function $f$, $K^mf$ is non-decreasing if $f$ is non-decreasing for all $m\ge 0$. Consider the adjoint chain $K^*$ of $K$ in $L^2(\pi)$. As birth and death chains are reversible, one has $K^*=K$. Using the identity $\mu K/\pi=K^*(\mu/\pi)$, it is easy to see that the above proof is consistent with the proof in \cite{DLP10}.
\end{rem}

\begin{proof}[Proof of Theorem \ref{t-sep}]
Assume that $K$ is irreducible and let $\lambda_0=0<\lambda_1<\cdots<\lambda_n$ be the eigenvalues of $I-K$ with $L^2(\pi)$-normalized eigenvector $\zeta_0=\mathbf{1},...,\zeta_n$. By Lemma \ref{l-bdc}, if $\mu$ satisfies $\mu(i)/\pi(i)\ge \mu(i+1)/\pi(i+1)$ for $0\le i<n$, then
\[
 d_{\text{sep}}^c(\mu,t)=1-\frac{\mu H_t(n)}{\pi(n)}=\sum_{j=1}^n\mu(\zeta_j)\zeta_j(n)e^{-\lambda_jt},
\]
where $\mu(\zeta_j)=\sum_{i=0}^n\zeta_j(i)\mu(i)$. If $K$ satisfies $p_i+q_{i+1}\le 1$ for all $0\le i<n$, then
\[
 d_{\text{sep}}(\mu,m)=1-\frac{\mu K^m(n)}{\pi(n)}=\sum_{j=1}^n\mu(\zeta_j)\zeta_j(n)(1-\lambda_j)^m.
\]
By Proposition \ref{p-zeta}, setting $\mu$ to be one of the dirac measure
$\delta_0,\delta_n$ leads to the desired identities.
\end{proof}

\section{Paths of infinite length}

In this section, the graph $G=(V,E)$ under consideration is infinite with $V=\{1,2,...\}$ and $E=\{\{i,i+1\}|i=1,2,...\}$. As before, let $\pi,\nu$ be positive measures on $V,E$ satisfying $\pi(V)=1$. The Dirichlet form and the variance are defined in a similar way as in the introduction and the spectral gap of $G$ with respect to $\pi,\nu$ is given by
\[
 \lambda^G_{\pi,\nu}=\inf\left\{\frac{\mathcal{E}_\nu(f,f)}{\text{Var}_\pi(f)}\bigg|
\text{$f$ is non-constant and $\pi(f^2)<\infty$}\right\}.
\]
For $n\ge 2$, let $G_n=(V_n,E_n)$ be the subgraph of $G$ with $V_n=\{1,2,...,n\}$, $E_n=\{\{i,i+1\}|1\le i<n\}$ and let $\pi_n,\nu_n$ be normalized restrictions of $\pi,\nu$ to $V_n,E_n$. That is, $\pi_n(i)=c_n\pi(i)$, $\nu_n(i,i+1)=c_n\nu(i,i+1)$ with $c_n=1/[\pi(1)+\cdots+\pi(n)]$. As before, let $M^G_{\pi,\nu}$ be an infinite matrix indexed by $V$ and defined by
\begin{equation}\label{eq-Minf}
 M^G_{\pi,\nu}(i,j)=-\frac{\nu(i,j)}{\pi(i)},\quad\forall |i-j|=1,\quad M^G_{\pi,\nu}(i,i)=\frac{\nu(i-1,i)+\nu(i,i+1)}{\pi(i)}.
\end{equation}
Clearly, $M^{G_n}_{\pi_n,\nu_n}$ is the principal submatrix of $M^G_{\pi,\nu}$ indexed by $V_n\times V_n$.

\begin{lem}\label{l-gapinf}
Referring to the above setting, $\lambda^{G_{n+1}}_{\pi_{n+1},\nu_{n+1}}<\lambda^{G_n}_{\pi_n,\nu_n}$ for $n>1$ and $\lambda^G_{\pi,\nu}=\lim_{n\ra\infty}\lambda^{G_n}_{\pi_n,\nu_n}$.
\end{lem}
\begin{proof}
Briefly, we write $\lambda$ for $\lambda^G_{\pi,\nu}$ and $\lambda_n$ for $\lambda^{G_n}_{\pi_n,\nu_n}$. Note that $\lambda_n$ is the smallest non-zero eigenvalue of the principal submatrix of $M^G_{\pi,\nu}$ indexed by $V_n\times V_n$. As a consequence of Proposition \ref{p-shape2}(1) and Remark \ref{r-shape}, $\lambda_{n+1}<\lambda_n$. For $n>1$, let $\phi_n$ be a minimizer for $\lambda_n$ and define $\psi_n(i)=\mathbf{1}_{V_n}(i)\phi_n(i)$ for $i\ge 1$. Clearly, one has $\mathcal{E}_{\nu_n}(\phi_n,\phi_n)=c_n\mathcal{E}_\nu(\psi_n,\psi_n)$ and $\text{Var}_{\pi_n}(\phi_n)=c_n\text{Var}_\pi(\psi_n)$. This implies $\lambda\le \lambda_n$ for $n\ge 2$. Let $\lambda^*=\lim_{n\ra\infty}\lambda_n$. Note that it remains to show $\lambda^*=\lambda$.
For $\epsilon>0$, choose a function $f$ on $V$ such that $\mathcal{E}_\nu(f,f)<(\lambda+\epsilon/2)\text{Var}_\pi(f)$ with $\pi(f^2)<\infty$. For $\delta>0$, we choose $N>0$ such that $\text{Var}_{\pi_N}(g)>(1-\delta)\text{Var}_\pi(f)$ and $\mathcal{E}_{\nu_N}(g,g)<(1+\delta)\mathcal{E}_\nu(f,f)$, where $g=f|_{V_N}$, the restriction of $f$ to $V_N$. This implies
\[
 \lambda^*\le\lambda_N\le\frac{\mathcal{E}_{\nu_N}(g,g)}{\text{Var}_{\pi_N}(g)}\le\frac{(1+\delta)\mathcal{E}_\nu(f,f)}{(1-\delta)\text{Var}_\pi(f)}.
\]
Letting $\delta\ra 0$ and then $\epsilon\ra 0$ yields $\lambda^*\le\lambda$, as desired.
\end{proof}
\begin{rem}
Silver \cite{S96} contains a discussion of the (weak*) convergence
of the spectral measure for $G_n$ to the spectral measure for $G$ in a very general setting. Lemma \ref{l-gapinf} can also be proved using Theorem 4.3.4 in \cite{S96}.
\end{rem}

\begin{prop}\label{p-shapeinf}
For $\lambda>0$, let $\phi_\lambda(1)=-1$ and
\[
 \phi_\lambda(i+1)=\phi_\lambda(i)+\frac{\{[\phi_\lambda(i)-\phi_\lambda(i-1)]\nu(i-1,i)-\lambda\pi(i)\phi_\lambda(i)\}^+}{\nu(i,i+1)},\,\forall i\ge 1.
\]
Set $\lambda_1=\infty$ and $\lambda_n=\lambda^{G_n}_{\pi_n,\nu_n}$ for $n\ge 2$.
\begin{itemize}
\item[(1)] For $i\ge 2$ and $\lambda\in[\lambda_i,\lambda_{i-1})$, $\phi_\lambda(i-1)<\phi_\lambda(i)=\phi_\lambda(i+1)$.

\item[(2)] For $\lambda\in(0,\lambda^G_{\pi,\nu}]$, $\phi_\lambda(i)<\phi_\lambda(i+1)$ for all $i\ge 1$.
\end{itemize}
\end{prop}
\begin{proof}
Immediate from Proposition \ref{p-shape2} and Remarks \ref{r-shape}-\ref{r-shape2}.
\end{proof}

\begin{rem}
By Proposition \ref{p-shapeinf}, one may generate a dichotomy algorithm for $\lambda^G_{\pi,\nu}$ using the shape of $\phi_\lambda$. See (\ref{alg-evi2}).
\end{rem}

The following theorem extends Theorem \ref{t-Miclo} to infinite paths.

\begin{thm}\label{t-ELinf}
If $\lambda^G_{\pi,\nu}>0$ and $\mathcal{E}_\nu(\psi,\psi)/\text{Var}_\pi(\psi)=\lambda^G_{\pi,\nu}$ for some function $\psi$ on $V$ with $\pi(\psi)=0$, then $\psi$ is strictly monotonic and satisfies
\[
 \lambda^G_{\pi,\nu}\pi(i)\psi(i)=[\psi(i)-\psi(i+1)]\nu(i,i+1)+[\psi(i)-\psi(i-1)]\nu(i-1,i),\,\,\forall i\ge 1.
\]
\end{thm}

\begin{thm}\label{t-gapinf}
For $\lambda>0$, let $\phi_\lambda$ be the function in Proposition \ref{p-shapeinf} and set $L(\lambda)=\mathcal{E}_\pi(\phi_\lambda,\phi_\lambda)/\textnormal{Var}_\pi(\phi_\lambda)$. Then,
\begin{itemize}
\item[(1)] $\lambda^G_{\pi,\nu}<L(\lambda)<\lambda$ for $\lambda\in(\lambda^G_{\pi,\nu},\infty)$.

\item[(2)] $L^n(\lambda)\ra\lambda^G_{\pi,\nu}$ as $n\ra\infty$ for $\lambda\in(\lambda^G_{\pi,\nu},\infty)$.
\end{itemize}
\end{thm}
\begin{proof}
Let $\lambda>\lambda^G_{\pi,\nu}$. By Lemma \ref{l-gapinf}, $\lambda_i\le \lambda<\lambda_{i-1}$ for some $i\ge 2$. By Proposition \ref{p-shapeinf} (1), one has $\phi_\lambda(i-1)<\phi_\lambda(i)=\phi_\lambda(i+1)$.  As in (\ref{eq-Enu}), we obtain
\[
 L(\lambda)=\lambda+\lambda\frac{\pi(\phi_\lambda)[\pi(\phi_\lambda)-\phi_\lambda(i)]}{\text{Var}_\pi(\phi_\lambda)},\quad \sum_{j=1}^i\phi_\lambda(j)\pi(j)\ge 0.
\]
This leads to $\pi(\phi_\lambda)>0$ and $\pi(\phi_\lambda)<\phi_\lambda(i)$, which implies $L(\lambda)<\lambda$. That means $L$ has no fixed point on $(\lambda^G_{\pi,\nu},\infty)$. The lower bound of (1) follows immediately from Theorem \ref{t-ELinf}. For (2), set $\lambda^*=\lim_{n\ra\infty}L^n(\lambda)\ge\lambda^G_{\pi,\nu}$. As a consequence of (1), $L$ is continuous on $(\lambda^G_{\pi,\nu},\infty)$. If $\lambda^*>\lambda^G_{\pi,\nu}$, then $\lambda^*$ is a fixed point of $L$, a contradiction! Hence, $\lambda^*=\lambda^G_{\pi,\nu}$.
\end{proof}

\section{A numerical experiment}

In this section, we illustrate the algorithm (\ref{alg-sp2}) on
 a specific Metropolis chain. The Metropolis algorithm introduced by Metropolis {\it et al.} in 1953 is a widely used construction that produces a Markov chain with a given stationary distribution $\pi$. Let $\pi$ be a positive probability measure on $V$ and $K$ be an irreducible Markov transition matrix on $V$. For simplicity, we assume that $K(x,y)=K(y,x)$ for all $x,y\in V$. The Metropolis chain evolves in the following way. Given the initial state $x$, select a state, say $y$, according to $K(x,\cdot)$ and compute the ratio $A(x,y)=\pi(y)/\pi(x)$. If $A(x,y)\ge 1$, then move to $y$. If $A(x,y)<1$, then flip a coin with probability $A(x,y)$ on heads and move to $y$ if the head appears. If the coin lands on tails, stay at $x$. Accordingly, if $M$ is the transition matrix of the Metropolis chain, then
\[
 M(x,y)=\begin{cases}K(x,y)&\text{if }A(x,y)\ge 1,\,x\ne y\\K(x,y)A(x,y)&\text{if }A(x,y)<1\\K(x,x)+\sum\limits_{z:A(x,z)<1}K(x,z)(1-A(x,z))&\text{if }x=y\end{cases}.
\]
It is easy to check $\pi(x)M(x,y)=\pi(y)M(y,x)$. As $K$ is irreducible, $M$ is irreducible. Moreover, if $\pi$ is not uniform, then $M(x,x)>0$ for some $x\in V$. This implies that $M$ is aperiodic and, consequently, $M^t(x,y)\ra\pi(y)$ and $e^{-t(I-M)}(x,y)\ra\pi(y)$ as $t\ra\infty$. For further information on Metropolis chains, see \cite{DS98} and the references therein.

For $n\ge 1$, let $G_n=(V_n,E_n)$ be a graph with $V_n=\{0,\pm 1,...,\pm n\}$ and $E_n=\{\{i,i+1\}:i=-n,...,n-1\}$. Suppose that $K_n$ is the transition matrix of the simple random walk on $V_n$, that is, $K_n(-n,-n)=K_n(n,n)=1/2$ and $K_n(i,i+1)=K_n(i+1,i)=1/2$ for all $-n\le i<n$. For $a>0$, let $\check{\pi}_{n,a},\hat{\pi}_{n,a}$ be probabilities on $V_n=\{0,\pm 1,...,\pm n\}$ given by
\[
 \check{\pi}_{n,a}(i)=\check{c}_{n,a}(|i|+1)^a,\quad \hat{\pi}_{n,a}(i)=\hat{c}_{n,a}(n-|i|+1)^a,
\]
where $\check{c}_{n,a}$ and $\hat{c}_{n,a}$ are normalizing constants. It is easy to compute that
\begin{equation}\label{eq-checkna}
 c_{n,a}/2\le1/\hat{c}_{n,a}<1/\check{c}_{n,a}\le 2c_{n,a},
\end{equation}
where
\[
 c_{n,a}=\frac{(n+1)^{a+1}}{a+1}+(n+1)^a.
\]
The Metropolis chains, $\check{K}_{n,a}$ and $\hat{K}_{n,a}$, for $\check{\pi}_{n,a}$ and $\hat{\pi}_{n,a}$ based on the simple random walk $K_n$ have transition matrices given by
\[
 \check{K}_{n,a}(i,j)=\check{K}_{n,a}(-i,-j),\quad \hat{K}_{n,a}(i,j)=\hat{K}_{n,a}(-i,-j)
\]
and
\[
 \check{K}_{n,a}(i,j)=\begin{cases}\frac{1}{2}&\text{if }j=i+1,i\in[0,n-1]\\\frac{i^a}{2(i+1)^a}&\text{if }j=i-1,i\in[1,n]\\\frac{(i+1)^a-i^a}{2(i+1)^a}&\text{if }j=i,i\notin\{0,n\}\\1-\frac{n^a}{2(n+1)^a}&\text{if }i=j=n\end{cases}
\]
and
\[
 \hat{K}_{n,a}(i,j)=\begin{cases}\frac{1}{2}&\text{if }j=i-1,i\in[1,n]\\\frac{(n-i)^a}{2(n-i+1)^a}&\text{if }j=i+1,i\in[0,n-1]\\\frac{(n-i+1)^a-(n-i)^a}{2(n-i+1)^a}&\text{if }j=i\ne 0\\1-\frac{n^a}{(n+1)^a}&\text{if }i=j=0\end{cases}.
\]

Saloff-Coste \cite{SC99} discussed the above chains and obtained the correct order of the spectral gaps. Let $\check{\lambda}_{n,a},\hat{\lambda}_{n,a}$ denote the spectral gaps of $\check{K}_{n,a},\hat{K}_{n,a}$. Referring to the recent work in \cite{CSal12-3}, one has
\[
 1/(4C)\le\lambda\le 1/C,
\]
where $(\lambda,C)$ is any of $(\check{\lambda}_{n,a},\check{C}_n(a))$ and $(\hat{\lambda}_{n,a},\hat{C}_n(a))$, and
\[
 \check{C}_n(a)=2\max_{1\le i\le n}\left(\sum_{j=0}^{i-1}(j+1)^{-a}\right)\left(\sum_{j=i}^n(j+1)^a\right),
\]
and
\[
 \hat{C}_n(a)=2\max_{1\le i\le n}\left(\sum_{j=0}^{i-1}(j+1)^a\right)\left(\sum_{j=i-1}^{n-1}(j+1)^{-a}\right).
\]

\begin{thm}\label{t-metropolis}
Let $\check{\lambda}_{n,a},\hat{\lambda}_{n,a}$ be spectral gaps for $\check{K}_{n,a},\hat{K}_{n,a}$. Then,
\[
 \frac{1}{8\eta_{-a}(1,n)\eta_a(2,n+1)}\le\check{\lambda}_{n,a}\le\frac{2}{\eta_{-a}(1,n)\eta_a(2,n+1)},
\]
and
\[
 \frac{1}{64\eta_a(1,\lceil n/2\rceil)\eta_{-a}(\lceil n/2\rceil,n)}\le\hat{\lambda}_{n,a}\le\frac{1}{2\eta_a(1,\lceil n/2\rceil)\eta_{-a}(\lceil n/2\rceil,n)}.
\]
where $\eta_a(k,l)=\sum_{i=k}^li^a$.
\end{thm}
\begin{proof}[Proof of Theorem \ref{t-metropolis}]
The bound for $\check{\lambda}_{n,a}$ follows immediately from the fact
\[
 \frac{\eta_{-a}(1,n)\eta_a(2,n+1)}{2}\le\check{C}_n(a)\le 2\eta_{-a}(1,n)\eta_a(2,n+1).
\]
For $\hat{\lambda}_{n,a}$, note that
\[
 \hat{C}_n(a)=2\max_{n/2\le i\le n}\left(\sum_{j=0}^{i-1}(j+1)^a\right)\left(\sum_{j=i-1}^{n-1}(j+1)^{-a}\right).
\]
Taking $i=\lceil n/2\rceil$ yields the upper bound. For the lower bound, we write
\[
 \hat{C}_n(a)=2\max_{n/2\le i\le n}\left(\sum_{j=0}^{i-1}\left(1-\frac{j}{i}\right)^a\right)
\left(\sum_{j=0}^{n-i}\left(1-\frac{j}{i+j}\right)^a\right).
\]
For $i\ge n/2$, it is clear that
\[
 \sum_{j=0}^{i-1}\left(1-\frac{j}{i}\right)^a\ge\sum_{j=0}^{i-1}\left(1-\frac{2j}{n}\right)^a
\ge\frac{1}{2}\sum_{j=0}^{n-1}\left(1-\frac{j}{n}\right)^a.
\]
Observe that, for $a>0$,
\begin{equation}\label{eq-c'ina}
 \frac{C'_{i,n}(a)}{2}\le\sum_{j=0}^{n-i}\left(1-\frac{j}{i+j}\right)^{a}\le C'_{i,n}(a),
\end{equation}
where
\[
 C'_{i,n}(a)=1+\begin{cases}i\frac{(i/n)^{a-1}-1}{1-a}&\text{if }a\ne 1\\i\log\frac{n}{i}&\text{if }a=1\end{cases}.
\]
It is clear that, for $i\ge n/2$, $C'_{i,n}(a)\le 2C'_{\lceil n/2\rceil,n}(a)$ and this leads to
\[
 \sum_{j=0}^{n-i}\left(1-\frac{j}{i+j}\right)^a\le 4\sum_{j=0}^{n-\lceil n/2\rceil}\left(1-\frac{j}{\lceil n/2\rceil+j}\right)^a.
\]
Summarizing all above gives the desired lower bound.
\end{proof}

\begin{figure}[ht]
\includegraphics[width=8cm]{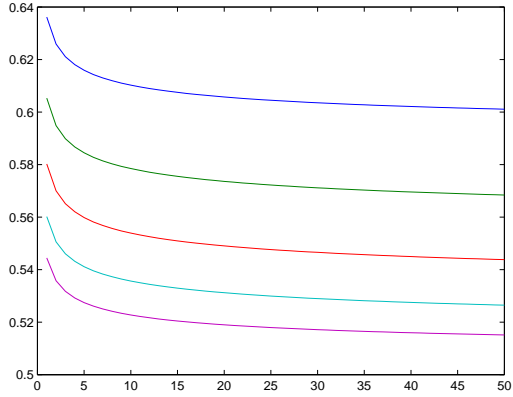}
\caption{These curves display the mapping
$m\mapsto\check{\lambda}_{100m,a}\eta_{-a}(1,100m)\eta_a(2,100m+1)$ in Theorem \ref{t-metropolis} in order from the top $a=0.8,0.9,1.0,1.1$ and $1.2$. The right most point corresponds to a path of length $n=5000$.}
\label{fig1}
\end{figure}

\begin{table}[ht]
\caption{These numbers denote $\check{\lambda}_{n,a}\eta_{-a}(1,n)\eta_a(2,n+1)$ in Theorem \ref{t-metropolis}.}
\begin{tabular}{cccccc}
\hline\hline
n & 10000& 20000 & 30000 & 40000 & 50000\\
\hline
a=0.8  &  0.5983  &  0.5960  &  0.5948  &  0.5941  &  0.5935\\
\hline
a=0.9  &  0.5652  &  0.5625  &  0.5610  &  0.5601  &  0.5594\\
\hline
a=1.0  &  0.5405  &  0.5377  &  0.5362  &  0.5353  &  0.5345\\
\hline
a=1.1  &  0.5235  &  0.5210  &  0.5197  &  0.5189  &  0.5183\\
\hline
a=1.2  &  0.5128  &  0.5109  &  0.5099  &  0.5093  &  0.5088\\
\hline\hline
\label{tab1}
\end{tabular}

\end{table}

\begin{rem}\label{r-sym}
Comparing with \cite[Theorem 9.5]{SC99}, the bounds for $\check{\lambda}_{n,a}$ given in Theorem \ref{t-metropolis} have a similar lower bound and an improved upper bound by a multiple of about $1/4$. For $\hat{\lambda}_{n,a}$, observe that
\[
 \frac{C''_{i}(a)}{2}\le\sum_{j=0}^{i-1}\left(1-\frac{j}{i}\right)^a\le C''_{i}(a),
\]
where
\[
 C''_{i}(a)=1+\frac{i-i^{-a}}{1+a}.
\]
Recall the constant $C'_{i,n}(a)$ in the proof of Theorem \ref{t-metropolis}. Note that
\[
 \frac{n+a}{2(1+a)}\le C''_{\lceil n/2\rceil}(a)\le\frac{2(n+a)}{(1+a)},
\]
and, for $a>0$, $a\ne 1$ and $n\ge 3$,
\[
 C'_{\lceil n/2\rceil,n}(a)\le 1+\frac{n+1}{2(1+a)}\sup_{a>0,a\ne 1}\frac{(2^{1-a}-1)(1+a)}{1-a}\le \frac{3(n+a)}{1+a},
\]
where the last inequality is obtained by considering the subcases $a<2$ and $a\ge 2$. The above computation also applies for $a=1$ and $n\in\{1,2\}$. In the same spirit, one can show that
$C'_{\lceil n/2\rceil,n}(a)\ge \frac{n+a}{6(1+a)}$. This yields
\begin{equation}\label{eq-cna}
 \frac{(n+a)^2}{6(1+a)^2}\le \hat{C}_{n,a}\le\frac{12(n+a)^2}{(1+a)^2},\quad\forall n\ge 1.
\end{equation}
Hence, we have $\hat{\lambda}_{n,a}\asymp(1+a)^2/(n+a)^2$. As a consequence of (\ref{eq-checkna}) and (\ref{eq-c'ina}), we obtain that, uniformly for $a>0$,
\[
 1/\check{\lambda}_{n,a}\asymp n^a\left(\left(1+\frac{1}{n}\right)^a+\frac{n}{1+a}\right)
\left(1+v(n,a)\right)\quad\text{as }n\ra\infty,
\]
where $v(n,1)=\log n$ and $v(n,a)=(n^{1-a}-1)/(1-a)$ for $a\ne 1$.
\end{rem}

\begin{rem}\label{r-sym2}
Note that the lower bound in Theorem \ref{t-lower} provides the correct order of the spectral gap for the chain $\check{K}_{n,a}$ uniformly in $a$ but not for $\hat{K}_{n,a}$. For instance, if $a$ grows with $n$, say $a=n$, then Theorem \ref{t-lower} implies $1/\hat{\lambda}_{n,n}=O(n)$, while (\ref{eq-cna}) gives $1/\hat{\lambda}_{n,n}\asymp 1$.
\end{rem}

\begin{rem}
Consider the chain in Theorem \ref{t-metropolis}. A numerical experiment of Algorithm (\ref{alg-sp2}) is implemented and the data is collected in Figure \ref{fig1} and Table \ref{tab1}. One may conjecture that $\check{\lambda}_{n,a}\eta_{-a}(1,n)\eta_a(2,n+1)\ra c(a)$ as $n\ra\infty$, where $c(a)$ is a constant depending on $a$.
\end{rem}

\section{Spectral gaps for uniform measures with bottlenecks}
In this section, we discuss some examples of special interests and show how the theory developed in the previous sections can be used to bound the spectral gap. In the first subsection, we develop a lower bound on the spectral gap in a very general setting using the theory in Section 3. In the second subsection, we focuses on the case of one bottleneck, where a precise estimation on the spectral gap is presented. Those computations are based on the theoretical work in Section 2. In the third subsection, we consider the case of multiple bottlenecks in which the exact order of the spectral gap is determined for some special classes of chains.

In what follows, we will use the notation $\pi(A)$ to represent the summation $\sum_{i\in A}\pi(i)$ for any measure $\pi$ on $V$ and any set $A\subset V$. Given two sequences of positive reals $a_n,b_n$, we write $a_n=O(b_n)$ if $a_n/b_n$ is bounded. If $a_n=O(b_n)$ and $b_n=O(a_n)$, we write $a_n\asymp b_n$. If $a_n/b_n\ra 1$, we write $a_n\sim b_n$.

\subsection{A lower bound on the spectral gap}
In this subsection, we give a lower bound on the spectral gap in the general case.
\begin{thm}\label{t-lower}
Let $G=(V,E)$ be a graph with vertex set $V=\{0,1,...,n\}$ and edge set $E=\{\{i,i+1\}|i=0,...,n-1\}$. Let $\pi,\nu$ be positive measures on $V,E$ with $\pi(V)=1$. Then,
\[
 \lambda^G_{\pi,\nu}\ge\max_{0\le i\le n}\left\{\left(\sum_{j=0}^{i-1}\frac{\pi([0,j])}{\nu(j,j+1)}\right)^{-1}\wedge\left(\sum_{j=i+1}^{n}
\frac{\pi([j,n])}{\nu(j-1,j)}\right)^{-1}\right\},
\]
where $a\wedge b:=\min\{a,b\}$.
\end{thm}
\begin{rem}\label{r-lower}
Let $C$ be the lower of the spectral gap in Theorem \ref{t-lower}. Note that, for any positive reals, $(a+b)/2\le \max\{a,b\}\le a+b$. Using this fact, it is easy to see that $C'\le C\le 2C'$, where
\[
 C'=\max_{0\le i\le n}\left(\sum_{j=0}^{i-1}\frac{\pi([0,j])}{\nu(j,j+1)}+\sum_{j=i+1}^{n}
\frac{\pi([j,n])}{\nu(j-1,j)}\right)^{-1}.
\]
In particular, if $i_0$ is the median of $\pi$, that is, $\pi([0,i_0])\ge 1/2$ and $\pi([i_0,n])\ge 1/2$, then
\[
 C'=\left(\sum_{j=0}^{i_0-1}\frac{\pi([0,j])}{\nu(j,j+1)}+\sum_{j=i_0+1}^{n}
\frac{\pi([j,n])}{\nu(j-1,j)}\right)^{-1}.
\]
\end{rem}

\begin{rem}
Let $(X_m)_{m=0}^\infty$ be an irreducible birth and death chain on $\{0,1,...,n\}$ with birth rate $p_i$, death rate $q_i$ and holding rate $r_i$ as in (\ref{eq-bdc}). For $0\le i\le n$, set $\tau_i=\min\{m\ge 0|X_m=i\}$ as the first passage time to state $i$. By the strong Markov property, the expected hitting time to $i$ started at $0$ can be expressed as
\[
\mathbb{E}_0\tau_i=\sum_{j=0}^{i-1}\frac{\pi([0,j])}{p_j\pi(j)},\quad\mathbb{E}_n\tau_i=\sum_{j=i+1}^n\frac{\pi([j,n])}{q_j\pi(j)},
\]
where $\pi$ is the stationary distribution of $(X_m)_{m=0}^\infty$. Let $\lambda$ be the spectral gap for $(X_m)_{m=0}^\infty$. Then, $\lambda=\lambda^G_{\pi,\nu}$, where $G$ is the path with vertex set $\{0,...,n\}$ and $\nu(i,i+1)=p_i\pi(i)=q_{i+1}\pi(i+1)$ for $0\le i<n$. The conclusion of Theorem \ref{t-lower} can be written as $1/\lambda\le\min_{0\le i\le n}\{\mathbb{E}_0\tau_i\vee\mathbb{E}_n\tau_i\}$.
\end{rem}

\begin{rem}
The lower bound in Theorem \ref{t-lower} is not necessary the right order of the spectral gap. See Remark \ref{r-sym2}.
\end{rem}

\begin{proof}[Proof of Theorem \ref{t-lower}]
For $\lambda>0$, let $\xi_\lambda$ be the function in Definition \ref{def-xil}. That is, $\xi_\lambda(0)=-1$ and, for $i\ge 0$,
\[
 [\xi_\lambda(i+1)-\xi_\lambda(i)]\nu(i,i+1)=[\xi_\lambda(i)-\xi_\lambda(i-1)]\nu(i-1,i)-\lambda\pi(i)\xi_\lambda(i).
\]
Inductively, one can show that if $1/\lambda>\sum_{j=0}^{\ell-1}[\pi([0,j])/\nu(j,j+1)]$, then
\[
 \begin{cases}0<\xi_\lambda(i+1)-\xi_\lambda(i)\le\lambda\pi([0,i])/\nu(i,i+1),\\-1\le\xi_\lambda(i+1)\le -1+\lambda\sum_{j=1}^{i}[\pi([0,j])/\nu(j,j+1)]<0,\end{cases}
\]
for $0\le i\le\ell-1$. One may do a similar computation from the other end point and, by Proposition \ref{p-twosided}, this implies
\[
 1/\lambda^{G_n}_{\pi_n,\nu_n}\le\max\left\{\sum_{j=0}^{\ell-1}\frac{\pi([0,j])}{\nu(j,j+1)},\sum_{j=1}^{n-\ell}\frac{\pi([n-j+1,n])}{\nu(n-j,n-j+1)}\right\}.
\]
Taking the minimum over $1\le \ell\le n$ gives the desired inequality.
\end{proof}

\subsection{One bottleneck}
For $n\ge 1$, let $G_n=(V_n,E_n)$ be the path on $\{0,1,...,n\}$ and set $\pi_n\equiv 1/(n+1)$ and $\nu_n\equiv 1/(n+1)$ with $C>0$. Using Feller's method in \cite[Chapter XVI.3]{F68}, one can show that the eigenvalues of $M^{G_n}_{\pi_n,\nu_n}$ are $2(1-\cos\frac{i\pi}{n+1})$ for $0\le i\le n$.

\begin{thm}\label{t-bottleneck}
For $n\ge 1$, let $\epsilon_n>0$, $1\le x_n\le \lceil n/2\rceil$ and set $\pi_n\equiv 1/(n+1)$,
\begin{equation}\label{eq-nu}
 \nu_n^{x_n}(x_n-1,x_n)=\frac{\epsilon_n}{n+1},\quad\nu_n^{x_n}(i-1,i)=\frac{1}{n+1},\quad\forall i\ne x_n.
\end{equation}
Then, the spectral gap are bounded by
\[
 \frac{1}{n^2/4+x_n/\epsilon_n}\le\lambda^G_{\pi_n,\nu_n^{x_n}}\le\min\left\{2\left(1-\cos\frac{\pi}{n-x_n+1}\right),\frac{\epsilon_n}{x_n}\right\}.
\]
In particular, $\lambda^{G_n}_{\pi_n,\nu_n^{x_n}}\asymp \min\{1/n^2,\epsilon_n/x_n\}$.
\end{thm}
\begin{proof}[Proof of Theorem \ref{t-bottleneck}]
The lower bound is immediate from Theorem \ref{t-lower} by choosing $i=\lceil n/2\rceil$ in the computation of the maximum. For the upper bound, we set $\lambda_n=1-\cos\frac{\pi}{n+1}$ and let $f_n$ be the function on $V_{n-x_n}$ defined by $f_n(0)=-1$ and, for $0\le i\le n-x_n-1$,
\[
 f_n(i+1)=f_n(i)+\frac{[f_n(i)-f_n(i-1)]\nu_{n-x_n}(i-1,i)-2\lambda_{n-x_n}\pi_{n-x_n}(i)f_n(i)}{\nu_{n-x_n}(i,i+1)}.
\]
By Proposition \ref{p-main1}, $\mathcal{E}_{\nu_{n-x_n}}(f_n,f_n)=2\lambda_{n-x_n}\text{Var}_{\pi_{n-x_n}}(f_n)$ and $\pi_{n-x_n}(f_n)=0$.
Let $g_n$ be the function on $V_n$ defined by $g_n(n-i)=f_n(i)$ for $0\le i\le n-x_n$ and $g_n(i)=f_n(n-x_n)$ for $0\le i<x_n$. A direct computation shows that
\[
 (n+1)\mathcal{E}_{\nu_n^{x_n}}(g_n,g_n)=(n-x_n+1)\mathcal{E}_{\nu_{n-x_n}}(f_n,f_n)
\]
and
\[
 (n+1)\text{Var}_{\pi_n}(g_n,g_n)=(n-x_n+1)\text{Var}_{\pi_{n-x_n}}(f_n)+\frac{x_n(n-x_n+1)}{n+1}f_n^2(n-x_n).
\]
This implies $\lambda^{G_n}_{\pi_n,\nu_n^{x_n}}\le 2\lambda_{n-x_n}$. On the other hand, using the test function, $h_n(i)=n-x_n+1$ for $0\le i<x_n$ and $h_n(i)=-x_n$ for $x_n\le i\le n$, one has
$\mathcal{E}_{\nu_n^{x_n}}(h_n,h_n)/\text{Var}_{\pi_n}(h_n)=\epsilon_n(n+1)/[x_n(n-x_n+1)]\le \epsilon_n/x_n$. This finishes the proof.
\end{proof}

The next theorem has a detailed description on the coefficient of the spectral gap. The proof is based on Section 3, particularly Proposition \ref{p-shape2} and Remark \ref{r-shape2}, and is given in the appendix.

\begin{thm}\label{t-bottleneck2}
For $n\ge 1$, let $x_n,\epsilon_n,\pi_n,\nu_n^{x_n}$ be as in Theorem \ref{t-bottleneck}. Suppose $x_n/(\epsilon_nn^2)\ra a\in[0,\infty]$ and $x_n/n\ra b\in[0,1/2]$.
\begin{itemize}
\item[(1)] If $a<\infty$ and $b=0$, then $\lambda^{G_n}_{\pi_n,\nu_n^{x_n}}\sim \min\{\pi^2,a^{-2}\}n^{-2}$.

\item[(2)] If $a<\infty$ and $b\in(0,1/2]$, then $\lambda^{G_n}_{\pi_n,\nu_n^{x_n}}\sim Cn^{-2}$, where $C$ is the unique positive solution of the following equation.
\[
 1+4\log 2-\frac{\pi^2}{6}-\frac{\pi^2aC}{1-b}-bC\sum_{i=1}^\infty\frac{(1-b)i^2-bC}{(i^2-C)[(1-b)^2i^2-b^2C]}=0.
\]

\item[(3)] If $a=\infty$, then $\lambda^{G_n}_{\pi_n,\nu_n^{x_n}}\sim \epsilon_n/x_n$.
\end{itemize}
\end{thm}

\subsection{Multiple bottlenecks}

In this subsection, we consider paths with multiple bottlenecks. As before, $G_n=(V_n,E_n)$ with $V_n=\{0,1,...,n\}$ and $E_n=\{\{i,i+1\}|i=0,...,n-1\}$. Let $k$ be a positive integer and $x_n=(x_{n,1},...,x_{n,k})$ be a $k$-vector satisfying $x_{n,i}\in V_n$ and $x_{n,1}\ge 1$ and $x_{n,i}<x_{n,i+1}$ for $1\le i<k$. Let $\epsilon_n=(\epsilon_{n,1},...,\epsilon_{n,k})$ be a vector with positive entries and $\nu_n^{x_n}$ be the measure on $E_n$ given by
\begin{equation}\label{eq-kbottleneck}
 \nu_n^{x_n}(i-1,i)=\begin{cases}1/(n+1)&\text{if }i\notin\{x_{n,1},...,x_{n,k}\}\\\epsilon_{n,j}/(n+1)&\text{if }i=x_{n,j},\,1\le j\le k\end{cases}.
\end{equation}

\begin{thm}\label{t-kbottleneck}
Let $G_n=(V_n,E_n)$ be the path on $\{0,...,n\}$. For $0\le k\le n$, let $\pi_n$ be the uniform probability on $V_n$ and $\nu_n^{x_n}$ be the measure on $E_n$ given by \textnormal{(\ref{eq-kbottleneck})}. Then,
\[
 \min\{1/(4n^2),C_{n,1}/2\}\le \lambda^{G_n}_{\pi_n,\nu_n^{x_n}}\le \min\left\{2\left(1-\cos\frac{\pi}{n-k+1}\right),C_{n,2}\right\},
\]
where
\[
 C_{n,1}=\left(\frac{n^2}{4}+\sum_{i=1}^k\min\{x_{n,i},n-x_{n,i}+1\}\left(\frac{1}{\epsilon_{n,i}}-1\right)\right)^{-1}
\]
and
\[
 C_{n,2}=\min_{0\le m_1\le m_2\le n}\left\{\frac{(n+1)\sum\limits_{i=m_1}^{m_2}1/\epsilon_{n,i}}{\sum\limits_{m_1\le i\le j\le m_2}
x_{n,i}(n-x_{n,j}+1)/(\epsilon_{n,i}\epsilon_{n,j})}\right\}.
\]
\end{thm}

\begin{rem}\label{r-kbottleneck}
Observe that, in Theorem \ref{t-kbottleneck}, $1-\cos\frac{2\pi}{n-k+1}\asymp n^{-2}$ and
\begin{align}
C_{n,2}&\le\min_{1\le j\le k}\left\{\frac{\epsilon_{n,j}}{\min\{x_{n,j},n-x_{n,j}+1\}}\right\}\notag\\
&=\min\left\{\min_{j:x_{n,j}\le \frac{n}{2}}\frac{\epsilon_{n,j}}{x_{n,j}},\min_{j:x_{n,j}>\frac{n}{2}}
\frac{\epsilon_{n,j}}{n-x_{n,j}+1}\right\}.\notag
\end{align}
\end{rem}

\begin{proof}[Proof of Theorem \ref{t-kbottleneck}]
We first prove the upper bound. Let $f_1$ be a function on $\{0,1,...,n\}$ satisfying $f(x_{n,j}-1)=f(x_{n,j})$ for $1\le i\le k$ and $f_2$ be a function on $\{0,...,n-k\}$ obtained by identifying points $x_{n,i}-1$ and $x_{n,i}$ for $1\le i\le k$. By setting $f_2$ as a minimizer for $\lambda^{G_{n-k}}_{\pi_{n-k},\nu_{n-k}}$ with $\pi_n(f_1)=0$, we obtain
\begin{align}
 2\left(1-\cos\frac{2\pi}{n-k+1}\right)&=\frac{\mathcal{E}_{\nu_{n-k}}(f_2,f_2)}{\text{Var}_{\pi_{n-k}}(f_2)}
\ge\frac{\mathcal{E}_{\nu_{n-k}}(f_2,f_2)}{\pi_{n-k}(f_2^2)}\notag\\
&\ge\frac{\mathcal{E}_{\nu_n}(f_1,f_1)}{\pi_n(f_1^2)}=\frac{\mathcal{E}_{\nu_n}(f_1,f_1)}{\text{Var}_{\pi_n}(f_1)}.\notag
\end{align}
To see the other upper bound, let $f_j$ be the function on $V_n$ satisfying $g_j(i)=-(n-x_{n,j}+1)$ for $0\le i\le x_{n,j}-1$ and $g_j(i)=x_{n,j}$ for $x_{n,j}\le i\le n$. Computations  show that $\pi_n(g_j)=0$, $\pi_n(g_ig_j)=x_{n,i}(n-x_{n,j}+1)$ for $i\le j$, and $\mathcal{E}_{\nu_n}(g_j,g_j)=\epsilon_{n,j}(n+1)$. Set $g=\sum_{j=1}^ka_jg_j$. As a consequence of the above discussion, we obtain
\[
 \frac{\mathcal{E}_{\nu_n}(g,g)}{\textnormal{Var}_{\pi_n}(g)}=\frac{(n+1)\sum_{i=1}^ka_i^2\epsilon_{n,i}}
{2\sum_{i<j}a_ia_jx_{n,i}(n-x_{n,j}+1)+\sum_{i=1}^ka_i^2x_{n,i}(n-x_{n,i}+1)}.
\]
Taking $a_i=1/\epsilon_{n,i}$ for $m_1\le i\le m_2$ and $a_i=0$ otherwise gives the bound $C_{n,2}$.

The lower bound is immediate from Theorem \ref{t-lower} and Remark \ref{r-lower}.
\end{proof}

Finally, we discuss some special cases illustrating Theorem \ref{t-kbottleneck}.

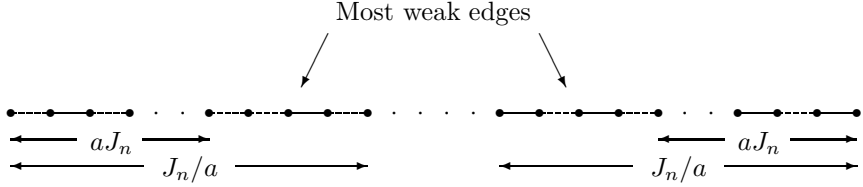
\begin{figure}[h]

\begin{center}
\caption{The dashed lines denote the weak edges of $\nu$ in Theorem \ref{t-unifp}. \label{fig2}}

\begin{picture}(160,80)(-10,-20)

\multiput(20,0)(275,0){2}{\multiput(-115,10)(15,0){4}{\circle*{3}}}
\multiput(20,0)(200,0){2}{\multiput(-60,10)(10,0){2}{\circle*{1}}}
\multiput(10,0)(110,0){2}{\multiput(-30,10)(15,0){5}{\circle*{3}}}
\multiput(50,10)(10,0){4}{\circle*{1}}

\put(10,0){\multiput(-105,10)(3,0){5}{\line(1,0){2}}
\put(-90,10){\line(1,0){15}}
\multiput(-75,10)(3,0){5}{\line(1,0){2}}}

\multiput(-20,10)(3,0){5}{\line(1,0){2}}
\multiput(-5,10)(3,0){5}{\line(1,0){2}}
\put(10,10){\line(1,0){15}}
\multiput(25,10)(3,0){5}{\line(1,0){2}}

\put(90,10){\line(1,0){15}}
\multiput(105,10)(3,0){5}{\line(1,0){2}}
\put(120,10){\line(1,0){15}}
\multiput(135,10)(3,0){5}{\line(1,0){2}}

\put(-10,0){\put(190,10){\line(1,0){15}}
\multiput(205,10)(3,0){5}{\line(1,0){2}}
\put(220,10){\line(1,0){15}}}

\multiput(20,0)(245,0){2}{\put(-90,0){\vector(-1,0){25}}\put(-65,0){\vector(1,0){25}}\put(-85,-4){$aJ_n$}}

\multiput(20,0)(185,0){2}{\put(-65,-10){\vector(-1,0){50}}\put(-30,-10){\vector(1,0){50}}\put(-58,-14){$J_n/a$}}

\put(25,40){\vector(-1,-2){10}}
\put(105,40){\vector(1,-2){10}}
\put(28,45){Most weak edges}

\end{picture}

\end{center}
\end{figure}

\begin{thm}\label{t-unifp}
For $n\ge 1$, let $\pi_n\equiv 1/(n+1)$ and $\nu_n$ be the measure in \textnormal{(\ref{eq-kbottleneck})} with $k_n$ bottlenecks satisfying $n-k_n\asymp n$. Suppose there are $I_n\subset\{1,...,k_n\}$, $a\in(0,1)$ and $J_n>0$ such that $|I_n|$ is bounded and, for $i\notin I_n$, $aJ_n\le \min\{x_{n,i},n-x_{n,i}+1\}\le J_n/a$. Then,
\[
 \lambda^{G_n}_{\pi_n,\nu_n}\asymp\min\left\{\frac{1}{n^2},\min_{i\in I_n}\frac{\epsilon_{n,i}}{\min\{x_{n,i},n-x_{n,i}+1\}},
\frac{\left(\sum_{i=1,i\notin I_n}^{k_n}1/\epsilon_{n,i}\right)^{-1}}{J_n}\right\}.
\]
\end{thm}
\begin{proof}
It is easy to get the lower bound from Theorem \ref{t-kbottleneck}, while the upper bound is the minimum of $C_{n,2}$ over all connected components of $\{1,...,\ell\}\setminus I_n$ and $\{\ell+1,...,k_n\}\setminus I_n$.
\end{proof}

See Figure \ref{fig2} for a reference on the bottlenecks. The following are immediate corollaries of Theorems \ref{t-kbottleneck}-\ref{t-unifp}.

\begin{cor}[Finitely many bottlenecks]\label{c-finitebtnk}
Referring to {\em Theorem \ref{t-unifp}}, if $k_n$ is bounded, then
\[
 \lambda^{G_n}_{\pi_n,\nu_n}\asymp\min\left\{\frac{1}{n^2},\min_{1\le i\le k_n}\frac{\epsilon_{n,i}}{\min\{x_{n,i},n-x_{n,i}+i\}}\right\}.
\]
\end{cor}

\begin{cor}[Bottlenecks far away the boundary]\label{c-farawaybtnk}
Referring to {\em Theorem \ref{t-unifp}}, if $n-k_n\asymp n$ and there are $a\in(0,1)$ and $J_n>0$ such that $aJ_n<\min\{x_{n,i},n-x_{n,i}+1\}<J_n/a$ for $1\le i\le k_n$, then
\[
 \lambda^{G_n}_{\pi_n,\nu_n}\asymp\min\left\{\frac{1}{n^2},\frac{\left(\sum_{j=1}^{k_n}1/\epsilon_{n,i}\right)^{-1}}{J_n}\right\}.
\]
\end{cor}

\begin{cor}[Uniformly distributed bottlenecks]\label{c-unifgap}
Referring to {\em Theorem \ref{t-unifp}}, if $\min_i\epsilon_{n,i}\asymp \max_i\epsilon_{n,i}$ and $x_{n,i}=\lfloor in/k_n\rfloor$ with $k_n\le n/2$, then
\[
 \lambda^{G_n}_{\pi_n,\nu_n}\asymp\min\left\{\frac{1}{n^2},\frac{\epsilon_{n,1}}{nk_n}\right\}.
\]
\end{cor}

\begin{rem}
Note that the assumption of the uniformity of $\pi$ and $\nu$,
except at the bottlenecks, can be relaxed by using a comparison argument.
\end{rem}

\appendix

\section{Techniques and proofs}
We start with an elementary lemma.

\begin{lem}\label{l-conv}
Let $a>0$ and $f:[a,\infty)\ra\mathbb{R}$ be a continuous function satisfying $f(a)=a$ and $f(x)\in[a,x)$ for $x>a$. For $b>a$, set $C_b=\sup_{a\le x\le b}\{(f(x)-a)/(x-a)\}$. Then, $C_b<1$ and $a\le f^n(b)\le a+C_b^n(b-a)$ for $n\ge 0$. Moreover, if $f$ is bounded on $[a,\infty)$, then $a\le f^n(x)\le a+C^n(x-a)$ for $n\ge 0$ and $x\ge a$ with $C=\sup_{a\le t<\infty}\{(f(t)-a)/(t-a)\}<1$.
\end{lem}

\begin{lem}\label{l-matt}
Let $(a_i,b_i,c_i)_{i=1}^\infty$ be sequences of reals with $b_i>0$ and $c_i>0$. For $n\ge 1$ and $t\in\mathbb{R}$, let
\[
    M_n(t)=\left(\begin{array}{cccccc}
    a_1-c_1t&1&0&0&\cdots&0\\
    b_1&a_2-c_2t&1&0&&\vdots\\
    0&b_2&\ddots&\ddots&\ddots&\vdots\\
    0&0&\ddots&\ddots&\ddots&0\\
    \vdots&&\ddots&\ddots&a_{n-1}-c_{n-1}t&1\\
    0&\cdots&\cdots&0&b_{n-1}&a_n-c_nt
    \end{array}\right).
\]
Then, there are $n$ distinct real roots for $\det M_n(t)=0$, say $t^{(n)}_1<\cdots<t^{(n)}_n$, and
\[
 t^{(n+1)}_j<t^{(n)}_j<t^{(n+1)}_{j+1},\quad\forall 1\le j\le n,\,n\ge 1.
\]
Furthermore, if $a_1\ge 1$ and $a_{i+1}\ge 1+b_i$, then $t^{(n)}_1>0$ for all $n\ge 1$.
\end{lem}

To prove Lemma \ref{l-matt}, we need the following statement.
\begin{lem}\label{l-mat}
Fix $n>0$ and, for $i\le 1\le n$, let $a_i,b_i,d_i$ be reals with $b_i>0$ and $d_i\ne 0$. Consider the following matrix
\begin{equation}\label{eq-mat}
    M=\left(\begin{array}{cccccc}
    a_1&d_1&0&0&\cdots&0\\
    d_1^{-1}b_1&a_2&d_2&0&&\vdots\\
    0&d_2^{-1}b_2&a_3&\ddots&\ddots&\vdots\\
    0&0&\ddots&\ddots&\ddots&0\\
    \vdots&&\ddots&\ddots&a_{n-1}&d_{n-1}\\
    0&\cdots&\cdots&0&d_{n-1}^{-1}b_{n-1}&a_n
    \end{array}\right).
\end{equation}
Then, the eigenvalues of $M$ are distinct reals and independent of $d_1,...,d_{n-1}$. Furthermore, if $a_1\ge 1$ and $a_{i+1}\ge 1+b_i$, then all eigenvalues of $M$ are positive.
\end{lem}
\begin{proof}[Proof of Lemma \ref{l-mat}]
Let $X,Y$ be diagonal matrices with $X_{11}=Y_{11}=1$, $X_{ii}=d_1d_1\cdots d_{i-1}$ and $Y_{ii}=(b_1b_2\cdots b_{i-1})^{-1/2}(d_1d_2\cdots d_{i-1})$ for $i>1$. One can show that
\[
 XMX^{-1}=\left(\begin{array}{cccccc}
    a_1&1&0&0&\cdots&0\\
    b_1&a_2&1&0&&\vdots\\
    0&b_2&a_3&\ddots&\ddots&\vdots\\
    0&0&\ddots&\ddots&\ddots&0\\
    \vdots&&\ddots&\ddots&a_{n-1}&1\\
    0&\cdots&\cdots&0&b_{n-1}&a_n
    \end{array}\right).
\]
Since $XMX^{-1}$ is independent of the choice of $d_1,...,d_{n-1}$, the eigenvalues of $M$ are independent of $d_1,...,d_{n-1}$. Note that $YMY^{-1}$ is Hermitian. This implies that the eigenvalues of $M$ are all real. As $M$ is tridiagonal with non-zero entries in the superdiagonal, the rank of $M-\lambda I$ is either $n-1$ or $n$. This implies that the eigenvalues of $M$ are all distinct.

Next, assume that $a_1\ge 1$ and $a_{i+1}\ge 1+b_i$. Let $(YMY^{-1})_i$ be the leading $i\times i$ principal matrices of $YMY^{-1}$. By induction, one can prove that $\det(YMY^{-1})_i=\prod_{j=1}^i\ell_j$, where $\ell_1=a_1$ and $\ell_{j+1}=a_{j+1}-b_j/\ell_j$ for $1\le j<n$. By the assumption at the beginning of this paragraph, $\ell_j\ge 1$ for all $1\le j<n$ and $\det(YMY^{-1})_i>0$ for all $1\le i\le n$. As the leading principal matrices have positive determinants, $(YMY^{-1})$ is positive definite. This proves that all eigenvalues of $M$ are positive.
\end{proof}

\begin{proof}[Proof of Lemma \ref{l-matt}]
We prove this lemma by induction. For $n=1$, it is clear that $t^{(1)}_1=a_1/c_1$ is the root for $\det M_1(t)$. For $n=2$, note that $\det M_2(t)$ is a quadratic function that tends to infinity as $|t|\ra\infty$. Since $\det M_2(t^{(1)}_1)=-b_1<0$, the polynomial, $\det M_2(t)$, has two real roots, say $t^{(2)}_1<t^{(2)}_2$, satisfying $t^{(2)}_1<t^{(1)}_1<t^{(2)}_2$. Now, we assume that, for some $n\ge 1$, $\det M_n(t)$ and $\det M_{n+1}(t)$ have reals roots $(t^{(n)}_i)_{i=1}^n$ and $(t^{(n+1)}_i)_{i=1}^{n+1}$ satisfying $t^{(n+1)}_i<t^{(n)}_i<t^{(n+1)}_{i+1}$ for $1\le i\le n$. Clearly, $\det M_n(t) \ra \infty$ as $t\ra-\infty$. This implies
\[
 \det M_n(t^{(n+1)}_{2k+2})<0<\det M_n(t^{(n+1)}_{2k+1}),\quad\forall k\ge 0.
\]
Observe that $\det M_{n+2}(t)=(a_{n+2}-c_{n+2}t)\det M_{n+1}(t)-b_{n+1}\det M_n(t)$. Replacing $t$ with $t^{(n+1)}_i$ yields
\[
 \det M_{n+2}(t^{(n+1)}_{2k+2})>0>\det M_{n+2}(t^{(n+1)}_{2k+1}),\quad\forall k\ge 0.
\]
This proves that $\det M_{n+2}(t)$ has $(n+2)$ distinct real roots with the desired interlacing property.

For the second part, assume that $a_1\ge 1$ and $a_{i+1}\ge 1+b_i$ for all $i\ge 1$. For $n=1$, it is obvious that $t^{(1)}_1>0$. Suppose $t^{(n)}_1>0$. According to the first part, we have $t^{(n+1)}_2>t^{(n)}_1>0$. By Lemma \ref{l-mat}, $\det M_{n+1}(0)>0$, which implies $t^{(n+1)}_1\ne 0$. As it is known that $\det M_{n+1}(t)<0$ for $t\in(t^{(n+1)}_1,t^{(n+1)}_2)$, it must be the case $t^{(n+1)}_1>0$. Otherwise, there will be another root for $\det M_{n+1}(t)$ between $t^{(n+1)}_1$ and $0$, which is a contradiction.
\end{proof}


\begin{proof}[Proof of Theorem \ref{t-bottleneck2}]
For convenience, we set $\lambda_n^m=1-\cos\frac{m\pi}{n+1}$ for $1\le m\le n$ and let $A_i(\lambda)$ be the $i$-by-$i$ tridiagonal matrix with entries $(A_i(\lambda))_{kl}=1$ for $|k-l|=1$ and $(A_i(\lambda))_{kk}=2-\lambda$. For $1\le j\le i$, let $B_i^j(\lambda)$ be the matrix equal to $A_i$ except the $(j,j)$-entry, which is defined by $(B_i^j(\lambda,\epsilon))_{jj}=2-\lambda/\epsilon$. By Remark \ref{r-shape}, $\lambda^{G_n}_{\pi_n,\nu_n^{x_n}}$ is the smallest root of $\det B_n^{x_n}(\lambda,\epsilon_n)=0$ and $(\lambda_{n,m})_{m=1}^n$ are roots of $\det A_n(\lambda)=0$. Note that, for $1\le j\le n$,
\[
 \frac{\det B_n^j(\lambda,\epsilon)}{\det A_{j-1}(\lambda)\det A_{n-j}(\lambda)}=\Delta_n^j(\lambda,\epsilon)=2-\lambda/\epsilon-R_{j-1}(\lambda)-R_{n-j}(\lambda),
\]
where $\det A_0(\lambda):=1$, $\det A_{-1}(\lambda):=0$ and
\[
 R_j(\lambda)=\frac{\det A_{j-1}(\lambda)}{\det A_j(\lambda)}=\frac{\prod_{i=1}^{j-1}(2\lambda_{j-1}^i-\lambda)}{\prod_{i=1}^j(2\lambda_j^i-\lambda)}.
\]
To prove this theorem, one has to determine the sign of $\Delta_n^j(\lambda,\epsilon)$.

Let $\ell_n=\delta_n/n^2$ with $\delta_n\ra 0$. As $n\ra\infty$,
\[
 \log\frac{2\lambda_n^i-\ell_n}{2\lambda_n^i}=-\frac{\delta_n}{2\lambda_n^i n^2}(1+o(1)),
\]
where $o(1)$ is uniform for $1\le i\le n$. Note that $\prod_{i=1}^j(2\lambda_j^i)=\det A_j(0)=j+1$. This implies
\begin{align}
 \log R_n(\ell_n)&=\log\frac{n}{n+1}+\left(\sum_{i=1}^n\frac{1}{\lambda_n^i n^2}
-\sum_{i=1}^{n-1}\frac{1}{\lambda_{n-1}^i(n-1)^2}\right)\frac{\delta_n(1+o(1))}{2}\notag\\
&=\log\frac{n}{n+1}+O\left(\frac{\delta_n}{n}\right).\notag
\end{align}
By a similar reasoning, one can prove that $\log R_j(\ell_n)=\log\frac{j}{j+1}+O(\delta_n/n)$ for bounded $j$.
This shows that, for $j_n\in\{1,...,n\}$ and $\ell_n=o(j_n^{-2})$,
\begin{equation}\label{eq-Rjl}
 R_{j_n}(\ell_n)=1-\frac{1}{j_n+1}+O(j_n\ell_n),\quad\text{as }n\ra\infty.
\end{equation}
Next, we compute $R_{j_n}(2C_n\lambda_{j_n}^1)$ with $C_n\ra C\in(0,1)$ and $j_n\ra\infty$. Note that, for $n$ large enough,
\begin{equation}\label{eq-Rj}
\begin{aligned}
 \log R_{j_n}(2C_n\lambda_{j_n}^1)
=&\sum_{i=1}^{j_n-1}\frac{\lambda_{j_n-1}^i-\lambda_{j_n}^i}{\lambda_{j_n}^i}-\frac{1}{2}
\sum_{i=1}^{j_n-1}\left(\frac{\lambda_{j_n-1}^i-\lambda_{j_n}^i}{\lambda_{j_n}^i}\right)^2\\
&\qquad+C_n\sum_{i=1}^{j_n-1}\frac{\lambda_{j_n}^1(\lambda_{j_n-1}^i-\lambda_{j_n}^i)}
{(\lambda_{j_n}^i-C_n\lambda_{j_n}^1)\lambda_{j_n}^i}-\log 4+O(j_n^{-2}).
\end{aligned}
\end{equation}
Calculus shows that
\begin{align}
 \sum_{i=1}^{j_n-1}\left(\frac{\lambda_{j_n-1}^i-\lambda_{j_n}^i}{\lambda_{j_n}^i}\right)^2&=\frac{1}{\pi j_n}\int_0^\pi\frac{\theta^2\sin^2\theta}{(1-\cos\theta)^2}d\theta+O(j_n^{-2})\notag\\&=\frac{8\log 2-\pi^2/3}{j_n}+O(j_n^{-2})\notag
\end{align}
and
\[
 \sum_{i=1}^{j_n-1}\frac{\lambda_{j_n}^1(\lambda_{j_n-1}^i-\lambda_{j_n}^i)}{(\lambda_{j_n}^i
-C\lambda_{j_n}^1)\lambda_{j_n}^i}=\frac{2}{j_n}\sum_{i=1}^\infty\frac{1}{i^2-C}+O(j_n^{-2}).
\]
Observe that, as $n\ra\infty$,
\[
 \log\frac{j_n}{j_n+1}=\log R_{j_n}(0)=\sum_{i=1}^{j_n-1}\frac{\lambda_{j_n-1}^i-\lambda_{j_n}^i}{\lambda_{j_n}^1}-\log 4+O(j_n^{-2}).
\]
Putting this back into (\ref{eq-Rj}) implies
\begin{equation}\label{eq-Rj2}
 R_{j_n}(2C_n\lambda_{j_n}^1)=1+\left(-1-4\log 2+\frac{\pi^2}{6}+C_n\sum_{i=1}^\infty\frac{1}{i^2-C_n}\right)\frac{1}{j_n}+O(j_n^{-2}).
\end{equation}
We consider the following two cases.

\noindent{\bf Case 1: $x_n=O(\epsilon_n n^2)$.} In this case, Theorem \ref{t-bottleneck} implies that $\lambda^{G_n}_{\pi_n,\nu_n^{x_n}}\asymp n^{-2}$. We assume further that $x_n/(\epsilon_nn^2)\ra a$ and $x_n/n\ra b$ with $a\in[0,\infty)$ and $b\in[0,1/2]$. Let $C_n\ra C\in(0,1)$. Replacing $j_n$ with $x_n-1$ in (\ref{eq-Rjl}) and with $n-x_n$ in (\ref{eq-Rj2}) yields that, for $b=0$,
\[
 \Delta_n^{x_n}(2C_n\lambda_{n-x_n}^1,\epsilon_n)=\frac{(1-\pi^2aC)(1+o(1))}{x_n}
\]
and, for $b\in(0,1/2]$,
\[
 \Delta_n^{x_n}(2C_n\lambda_{n-x_n}^1,\epsilon_n)=\left(1+4\log 2-\frac{\pi^2}{6}-\frac{\pi^2aC}{1-b}-bC\kappa_b(C)\right)\frac{(1+o(1))}{b(1-b)n},
\]
where $\kappa_t(c)=\sum_{i=1}^\infty\frac{(1-t)i^2-tc}{(i^2-c)[(1-t)^2i^2-t^2c]}$. This proves (1) and (2).

\noindent{\bf Case 2: $\epsilon n^2=o(x_n)$.} This is exactly (3) and the result is immediate from Theorem \ref{t-bottleneck}.
\end{proof}


\begin{thebibliography}{10}

\bibitem{BS87}
M.~Brown and Y.-S. Shao.
\newblock Identifying coefficients in the spectral representation for first
  passage time distributions.
\newblock {\em Probab. Engrg. Inform. Sci.}, 1:69--74, 1987.

\bibitem{CSal08}
Guan-Yu Chen and Laurent Saloff-Coste.
\newblock The cutoff phenomenon for ergodic markov processes.
\newblock {\em Electron. J. Probab.}, 13:26--78, 2008.

\bibitem{CSal10}
Guan-Yu Chen and Laurent Saloff-Coste.
\newblock The {$L^2$}-cutoff for reversible {M}arkov processes.
\newblock {\em J. Funct. Anal.}, 258(7):2246--2315, 2010.

\bibitem{CSal12-3}
Guan-Yu Chen and Laurent Saloff-Coste.
\newblock On the mixing time and spectral gap for birth and death chains.
\newblock In preparation, 2012.

\bibitem{DS98}
P.~Diaconis and L.~Saloff-Coste.
\newblock What do we know about the {M}etropolis algorithm?
\newblock {\em J. Comput. System Sci.}, 57(1):20--36, 1998.
\newblock 27th Annual ACM Symposium on the Theory of Computing (STOC'95) (Las
  Vegas, NV).

\bibitem{DF90}
Persi Diaconis and James~Allen Fill.
\newblock Strong stationary times via a new form of duality.
\newblock {\em Ann. Probab.}, 18(4):1483--1522, 1990.

\bibitem{DS93-1}
Persi Diaconis and Laurent Saloff-Coste.
\newblock Comparison techniques for random walk on finite groups.
\newblock {\em Ann. Probab.}, 21(4):2131--2156, 1993.

\bibitem{DS93-2}
Persi Diaconis and Laurent Saloff-Coste.
\newblock Comparison theorems for reversible {M}arkov chains.
\newblock {\em Ann. Appl. Probab.}, 3(3):696--730, 1993.

\bibitem{DS06}
Persi Diaconis and Laurent Saloff-Coste.
\newblock Separation cut-offs for birth and death chains.
\newblock {\em Ann. Appl. Probab.}, 16(4):2098--2122, 2006.

\bibitem{DLP10}
Jian Ding, Eyal Lubetzky, and Yuval Peres.
\newblock Total variation cutoff in birth-and-death chains.
\newblock {\em Probab. Theory Related Fields}, 146(1-2):61--85, 2010.

\bibitem{F68}
William Feller.
\newblock {\em An introduction to probability theory and its applications.
  {V}ol. {I}}.
\newblock Third edition. John Wiley \& Sons Inc., New York, 1968.

\bibitem{F92}
James~Allen Fill.
\newblock Strong stationary duality for continuous-time {M}arkov chains. {I}.
  {T}heory.
\newblock {\em J. Theoret. Probab.}, 5(1):45--70, 1992.

\bibitem{GK37}
F.~R. Gantmacher and M.~G. Krein.
\newblock Sur les matrices compl\'{e}tement non n\'{e}gatives et oscillatoires.
\newblock {\em Compositio Math.}, 4:445--470, 1937.

\bibitem{M08}
Laurent Miclo.
\newblock On eigenfunctions of {M}arkov processes on trees.
\newblock {\em Probab. Theory Related Fields}, 142(3-4):561--594, 2008.

\bibitem{M09}
Laurent Miclo.
\newblock Monotonicity of the extremal functions for one-dimensional
  inequalities of logarithmic {S}obolev type.
\newblock In {\em S\'eminaire de {P}robabilit\'es {XLII}}, volume 1979 of {\em
  Lecture Notes in Math.}, pages 103--130. Springer, Berlin, 2009.

\bibitem{SC99}
L.~Saloff-Coste.
\newblock Simple examples of the use of {N}ash inequalities for finite {M}arkov
  chains.
\newblock In {\em Stochastic geometry ({T}oulouse, 1996)}, volume~80 of {\em
  Monogr. Statist. Appl. Probab.}, pages 365--400. Chapman \& Hall/CRC, Boca
  Raton, FL, 1999.

\bibitem{S96}
Jeffrey~Scott Silver.
\newblock {\em Weighted {P}oincare and exhaustive approximation techniques for
  scaled {M}etropolis-{H}astings algorithms and spectral total variation
  convergence bounds in infinite commutable {M}arkov chain theory}.
\newblock ProQuest LLC, Ann Arbor, MI, 1996.
\newblock Thesis (Ph.D.)--Harvard University.

\bibitem{T96}
Gerald Teschl.
\newblock Oscillation theory and renormalized oscillation theory for {J}acobi
  operators.
\newblock {\em J. Differential Equations}, 129(2):532--558, 1996.

\bibitem{T00}
Gerald Teschl.
\newblock {\em Jacobi operators and completely integrable nonlinear lattices},
  volume~72 of {\em Mathematical Surveys and Monographs}.
\newblock American Mathematical Society, Providence, RI, 2000.

\end{thebibliography}

\end{document}